\documentclass{ws-ijnt} 

\usepackage{pdflscape}
\usepackage{caption}
\usepackage{cite}

\newtheorem{alg}[theorem]{ALGORITHM}

\DeclareMathOperator{\Pic}{Pic}
 \DeclareMathOperator{\Div}{Div}
 \DeclareMathOperator{\To}{T^0}
 
 \DeclareMathOperator{\Tr}{Tr}  
  \DeclareMathOperator{\co}{covol}  
  
  \DeclareMathOperator{\pol}{poly} 

\begin{document}

\markboth{Ha Thanh Nguyen Tran}
{Computing dimensions of spaces of Arakelov divisors of number fields}

%
\catchline{}{}{}{}{}
%

\title{COMPUTING DIMENSIONS OF SPACES OF ARAKELOV DIVISORS OF NUMBER FIELDS
}

\author{HA THANH NGUYEN TRAN}

\address{Department of Mathematics and Systems Analysis,\\ Aalto University School of Science,\\ Otakaari 1, 
Espoo 02150, Finland.\\
\email{hatran1104@gmail.com} }

\maketitle

\begin{history}
\received{(16 June 2015)}
\accepted{(19 February 2016)}
\end{history}

\begin{abstract}
The function $h^0$ for a number field is analogous to the dimension of the Riemann-Roch spaces at divisors on an algebraic curve. We provide a method to compute this function for number fields with unit group of rank at most 2, even with large discriminant. This method is based on using LLL-reduced bases, the ``jump algorithm" and Poisson summation formula.
\end{abstract}



\section{Introduction}
As an analogue of the dimension of the Riemann-Roch spaces of divisors on an algebraic curve, van der Geer and Schoof introduced the function $h^0$ for a number field $F$ (see \cite{ref:3}). This function is also called the ``size function" for $F$ (see \cite{ref:14,ref:15,ref:27,ref:21,ref:38}). The properties and an upper bound for $h^0$ were provided in \cite{ref:27,ref:21,ref:3}.
 After that, in {\cite[Section 10.9]{ref:4}},  Schoof proposed a method to compute this function by using reduced Arakelov divisors. Essentially, we can approximate the value of $h^0$ at a given class of divisor $D$ by knowing the short vectors of the lattice $L$ associated to $D$ because the main contributions to $h^0(D)$ come from the shortest vectors of $L$ (see \cite{ref:15} and \cite{ref:3}). This can be efficiently computed if $L$ comes with a good, i.e., reasonably orthogonal, basis. 

Here we present a method to approximate the value of $h^0$ by using  the Poisson summation formula  as well as some ``good" divisors which can be obtained  by doing the ``jump algorithm" {\cite[Algorithm 10.8]{ref:4}} and from an LLL-reduction on $D$ (see Section \ref{sec:4} and \cite{ref:35,ref:33}). These divisors may not be reduced in the usual sense (see \cite{ref:4}) but they can be used to compute $h^0$ efficiently.

Let $F$ be a number field of degree $n$ and discriminant $\Delta_F$ with unit group of rank at most 2. We compute an approximate value of $h^0$ at any class of Arakelov divisors on $\Pic_F$. This method does not require a basis of the unit group of the ring of integers $O_F$ of $F$ and runs in polynomial time in $\log{|\Delta_F|}$ (see Section \ref{sec:runtime}).

In Section 2, we give a brief introduction to Arakelov divisors, the Arakelov class group and the group $\Pic_F$, the function $h^0$ for a number field and the Poisson summation formula for lattices. Then we discuss some results on  Arakelov divisors obtained from the LLL-algorithm in Section 3. Section 4 provides an algorithm to approximate the values of $h^0$ for number fields with large discriminant. Section 5 is devoted to bounding the error as well determining the running time of the algorithm. We also present some numerical examples applying this method to compute $h^0$ for real quadratic fields and real cubic fields in Section 6.

\section{Preliminaries}\label{pre}
This section briefly recalls some basic definitions that are used in the next sections. See \cite{ref:4, ref:3} for full details. 

Let $F$ be a number field of degree $n$ with the ring of integers $O_F$ and $r_1, r_2$ the number of real and complex infinite primes (or infinite places). Let $Cl_F$ be the class group of $F$. Denote by $F_{\mathbb{R}}: = F \otimes_\mathbb{Q}\mathbb{R} \simeq \prod_{\sigma \text{ real}}\mathbb{R} \times \prod_{\sigma \text{ complex}}\mathbb{C}$ where $\sigma$ runs through the infinite primes of $F$. 
 Then $F_{\mathbb{R}}$ is an \'{e}tale $\mathbb{R}$-algebra with the canonical Euclidean structure given by the scalar product\\
$\hspace*{3cm}\langle u, v \rangle := \Tr(u \overline{v}) $ \text{ for } $u = (u_{\sigma}), v = (v_{\sigma}) \in F_{\mathbb{R}}$.\\
In particular, in terms of coordinates, we have
 $$\|u\|^2= \Tr(u \overline{u})= \sum_{\sigma } \deg \sigma |u_{\sigma}|^2, \text{ for any } u = (u_{\sigma}) \in F_{\mathbb{R}}.$$
Here the degree of an infinite prime $\sigma$ is equal to 1 or 2 depending on whether $\sigma$ is real or complex.\\
The \textit{norm} of an element $u = (u_{\sigma})_{\sigma} $ of $F_{\mathbb{R}}$ is defined by 
$N(u):= \prod_{\sigma \text{ real} }u_{\sigma} \cdot \prod_{\sigma \text{ complex} } |u_{\sigma}|^2.$

\subsection{Arakelov divisors}\label{sec:ara}
Here and in the rest of the paper we often call fractional ideals
simply `ideals'. If we want to emphasize that an ideal is integral, we call it an integral ideal. 
 \begin{definition}\label{DefAra}
An \textit{Arakelov divisor} is a pair $D=(I,u)$ where $I$ is an ideal and $u$ is an arbitrary unit in $\prod_{\sigma}\mathbb{R}^*_{+} \subset F_{\mathbb{R}}$. The Arakelov divisors of $F$ form an additive group denoted by $\Div_F$. 
\end{definition}

Let $D = (I,u)$ be an Arakelov divisor. The \textit{degree} of $D$ is defined by $deg(D):= -\log{\left(N(u) N(I)\right)}$. We associate to $D$ the \textit{lattice} $ uI:= \{u f = (u_{\sigma} \cdot \sigma(f))_{\sigma}: f \in I \} \subset F_{\mathbb{R}} $ with the inherited  metric from $F_{\mathbb{R}}$ (see \cite{ref:0,ref:4}). For each $f \in I$,  by putting $\| f\|_D:= \| uf\|$, we obtain a  scalar product on $I$ that makes $I$ become an ideal lattice as well {\cite[Section 4]{ref:4}}.   
 The covolume of the lattice $L= u I$ associated to $D$ is $\co(L) = \sqrt{|\Delta_F|}\hspace*{0.1cm} N(I) N(u) = \sqrt{|\Delta_F|}\hspace*{0.1cm} e^{-deg(D)}$.

To each element $f \in F^*$ is attached a \textit{principal} Arakelov divisor $(f)=(f^{-1} O_F, |f|)$ where 
$f^{-1} O_F$ is the principal ideal generated by $f^{-1}$ and $|f| = (|\sigma(f)|)_{\sigma} \in \prod_{\sigma}\mathbb{R}^*_{+} \subset F_{\mathbb{R}}$. It has degree $0$ by the product formula.
\subsection{The group $\Pic_F$ and the Arakelov class group $\Pic_F^0$} 

\begin{definition}
The quotient of $\Div_F$ by its subgroup of principal Arakelov divisors is denoted by $\Pic_F$.
\end{definition}
The set of all Arakelov divisors of degree 0 form a subgroup of $\Div_F$, denoted by $\Div_F^0$. The Arakelov class group of a number field is analogous to the Picard group of an algebraic curve defined as follows.
\begin{definition}
The \textit{Arakelov class group}  $\Pic_F^0$ is the quotient of $\Div_F^0$ by its subgroup  of principal divisors.
\end{definition}
 
Each $v=(v_{\sigma})  \in \oplus_{\sigma}\mathbb{R}$ can be embedded into $\Div_F$ as the divisor $D_v= (O_F, u)$ with $u = (e^{-v_{\sigma}})_{\sigma}$.

Put 
$(\oplus_{\sigma}\mathbb{R})^0 = \{ (v_{\sigma}) \in \oplus_{\sigma}\mathbb{R}: deg(D_v) = 0 \}$ and $\Lambda = \{(\log|\sigma(\varepsilon)|)_{\sigma}: \varepsilon \in O_F^*\}$. Then $\Lambda$ is a lattice contained in  the vector space $(\oplus_{\sigma}\mathbb{R})^0$. We define
$$\To = (\oplus_{\sigma}\mathbb{R})^0 / \Lambda.$$
By Dirichlet's unit theorem, $\To$ is a compact real torus of dimension $r_1+r_2-1$ {\cite[Section 4.9]{ref:11}}. 
The structure of $\Pic_F^0$ is described by the following proposition.
\begin{proposition}\label{prop:structure1}
The map that sends the class of a divisor  $(I, u)$ to the  class of the ideal $I$  is a homomorphism 
from $ \Pic_F^0$ to the class group $Cl_F$ of $F$. 
It induces the exact sequence 
\begin{align*}
 0 \longrightarrow \To \longrightarrow \Pic_F^0 \longrightarrow Cl_F \longrightarrow 0  .
\end{align*}
\end{proposition}
\begin{proof}
 See {\cite[Proposition 2.2]{ref:4}}.
\end{proof}
\subsection{Metric on the Arakelov class group}\label{sec:metric}

For $u \in \prod_{\sigma}\mathbb{R}^*_{+} $, we let $\log{u}$ denote the element 
$\log u: = (\log(u_{\sigma}))_{\sigma} \in \prod_{\sigma}\mathbb{R} \subset F_\mathbb{R }.$
By using the scalar product from $F_\mathbb{R }$, this vector has length
$\|\log u\|^2 = \sum_{\sigma } deg(\sigma)\log^2(u_{\sigma}).$
We define
\begin{equation*}
  \|u\|_{Pic}: =
  \min_{\substack{u' \in \prod_{\sigma}{\mathbb{R}_{>0}}\\
                  \log u' \equiv  \log u (\text{ mod } \Lambda)}}
        \|\log u'\| = \min_{\varepsilon \in O_F^*}    \|\log(|\varepsilon| u)\|. 
\end{equation*}

Now let $[D]$ and $[D']$ be two classes containing divisor $D$ and $D'$ respectively lying on the same connected component of $\Pic_F^0$. Then by Proposition \ref{prop:structure1}, there is some unique  $ u \in \To$ such that $D - D' = (O_F, u)$. We define the \textit{distance} between two divisor classes containing $D$ and $D'$ to be $\|u\|_{Pic}$.
The function $\|  \hspace{0.3cm}\|_{Pic}$ gives rise to a distance function that induces the natural topology of $\Pic_F^0$.
See Section 6 in \cite{ref:4} for more details.
\subsection{The function $h^0$ for a number field}
 Let $D=(I,u)$ be an Arakelov divisor of $F$. The \textit{effectivity} $e(D)$ of $D$ is the number defined by
 \[ 
 e(D) = \begin{cases}
 0 & \text{ if }  O_F \nsubseteq I\\
 e^{-\pi\|1\|^2_{D}} & \text{ if } O_F \subseteq I.
 \end{cases}
 \]
 This number is between $0$ and $1$. A divisor $D$ is call \textit{effective} if $e(D)>0$. Similar to a Riemann-Roch space of an algebraic curve, we denote by $H^0(D)$ the union of $\{0\}$ and the set of all elements $f$ of $F$ for which the divisor $(f) + D$ is effective, i.e.,
$$ H^0(D): = \{f \in F^* : e((f) + D) > 0\} \cup \{0\}.$$

Since $e((f) + D) >0$ if and only if $ f \in I\setminus \{0\}$, the set $H^0(D)$ is equal to the infinite group $I$. In order to measure its size, we weight each element $f \in I$ with the effectivity of the Arakelov divisor  $(f)+D$:
$$e((f)+D)=e^{-\pi\|1\|_{(f)+D}^2}=e^{-\pi\|f\|_D^2}.$$
 The value of the function $h^0$ at $D$ is obtained by summing up these terms for all $f$ in $I$ including $0$ and then taking its logarithm as follows.
 $$ h^0(D)=\log{\left(\sum_{f \in I}e^{-\pi\|f\|^2_{D}}\right)}=\log{\left(\sum_{f \in I}e^{-\pi\|f u\|^2}\right)}.$$
 Since two Arakelov divisors in the same class in $\Pic_F$ have isometric associated lattices, the function $h^0 (D)$ only depends on the class $[D]$ of $D$ in $\Pic_F$ and we may write $h^0([D])$. In other words, $h^0$ is well defined on $\Pic_F$. See \cite{ref:3} for more details.
\subsection{The Poisson summation formula for lattices}\label{Poisson}
Let $L$ be a lattice  in $\mathbb{R} ^n$ and $L=L_1 \oplus L_2$. Note that we do not require that $L_1$ and $L_2$ are orthogonal to one another.

Consider the sum 
$$S = \sum_{\mathsf{z} \in L} e^{-\pi \|\mathsf{z}\|^2}.$$
Let $V$ be the subspace $L_1 \otimes \mathbb{R}$ of $\mathbb{R}^n$ and let $\pi(\mathsf{b})$ denote the orthogonal projection onto $V$ of a vector $\mathsf{b} \in \mathbb{R}^n$. We have the following.

\begin{lemma}\label{poi1}
Let $\gamma$ be the covolume of the lattice $L_1$ inside $V$. Then 
$$S = \frac{1}{\gamma}\sum_{\mathsf{b} \in L_2}e^{-\pi \|\mathsf{b}-\pi(\mathsf{b})\|^2}\sum_{\mathsf{a} \in L_1^{\vee}} e^{-\pi \|\mathsf{a}\|^2 -2 \pi i \langle \mathsf{a} , \pi(\mathsf{b})\rangle } .$$
\end{lemma}

\begin{proof}
Let $\mathsf{z} \in L$. Then $\mathsf{z} = \mathsf{a} + \mathsf{b} $ for some $\mathsf{a} \in L_1$ and $\mathsf{b} \in L_2$. The vectors $\mathsf{b} - \pi(\mathsf{b})$ and $ \mathsf{a} + \pi(\mathsf{b})$ are  orthogonal. The Pythagorean Theorem implies therefore
$$\|\mathsf{a} +\mathsf{b}\|^2 = \|\mathsf{b}-\pi(\mathsf{b})\|^2 + \|\mathsf{a} + \pi(\mathsf{b})\|^2.$$

Thus, we can write
$$S= \sum_{\mathsf{b} \in L_2}\sum_{\mathsf{a} \in L_1} e^{-\pi \|\mathsf{a}+\mathsf{b}\|^2}
 =\sum_{\mathsf{b} \in L_2}e^{-\pi \|\mathsf{b}-\pi(\mathsf{b})\|^2}\sum_{\mathsf{a} \in L_1} e^{-\pi \|\mathsf{a} + \pi(\mathsf{b})\|^2}.$$

Applying the Poisson summation formula to the second sum (see \cite{ref:37} and Chapter 7 in \cite {ref:29}), we obtain that

$$S = \frac{1}{\gamma}\sum_{\mathsf{b} \in L_2}e^{-\pi \|\mathsf{b}-\pi(\mathsf{b})\|^2}\sum_{\mathsf{a} \in L_1^{\vee}} e^{-\pi \|\mathsf{a}\|^2 -2 \pi i \langle \mathsf{a} , \pi(\mathsf{b})\rangle } .$$
\end{proof}

\section{Some Results}\label{result}
In this section, we discuss ``nice" properties of Arakelov divisors obtained from the LLL-algorithm. 

From now on, we put 
$$\left(\prod_{\sigma} \mathbb{R}_{+}^*\right)^0 = \left\{ s   \in \prod_{\sigma} \mathbb{R}_{+}^*: N(s) = 1\right\} $$
$$\text{ and} \hspace*{1.2cm} \partial_F = \left(\frac{2}{\pi}\right)^{r_2}\sqrt{|\Delta_F|} \text{  and  } \mathcal{D}_F =  \left(2^{(n-1)/2} \sqrt{n}\right)^n \partial_F.\hspace*{0.5cm} $$ 

 \begin{definition}
 Let $J$ be an ideal of $F$. Then \textit{the Arakelov divisor associated to} $J$ is $d(J):=(J, u)$ where $u = (u_{\sigma}) \in \prod_{\sigma}\mathbb{R}^*_{+}$ and $u_{\sigma} = N(J)^{-1/n}$ for all $\sigma$.	
 \end{definition}
 
 Let $D =(I,u)$ be an Arakelov divisor and let $L = uI$ be the lattice associated to $D$. Assume that a basis of $L$ is given. By using the LLL-algorithm, we can find an LLL-reduced basis $\{b_1, ..., b_n\}$ of $L$. Since $b_1 \in L =u I$, there is some nonzero element $f \in I$ such that $b_1 = u\cdot f$. Denote by 
 $$J = b_1^{-1} L = u^{-1} f^{-1} (uI) = f^{-1} I.$$
 Then $J$ is an ideal of $F$. Therefore, we can define as follows.
\begin{definition}
Let $D =(I,u)$ be an Arakelov divisor and let $L = uI$ be the lattice associated to $D$ with a known basis. We call \textit{an LLL-reduction on $D$} the process of finding an LLL-reduced basis $\{ b_1, ..., b_n\} $ of $L$, then computing a new ideal lattice $J = b_1^{-1}L $ and a new divisor $D' = d(J)$.
\end{definition}

We first recall the following lemma {\cite[Proposition 4.4]{ref:4}}.
\begin{lemma}\label{lem:min}
Let 
$D = (I,u)$ be a divisor of degree $0$. Then there is a nonzero element $f \in I$ such that 
$ u_{\sigma} |\sigma(f)| \leq \partial_F^{1/n} \text{ for all } \sigma$. In particular $\| u f\| \leq \sqrt{n}\hspace*{0.1cm} \partial_F^{1/n}$.
\end{lemma}
\begin{proof}
See {\cite[Proposition 4.4]{ref:4}}.
\end{proof}

We prove the proposition below.

\begin{proposition}\label{pro:close} 
Let $D=(I,u)$ be an Arakelov divisor of degree $0$ and $D' = d(J)$ obtained  by an LLL-reduction on $D$.
Then we have the following.
\begin{itemize}
\item [i)] The ideal $J^{-1}$ is integral with $N(J^{-1}) \le 2^{n(n-1)/2} \partial_F$.\\ 
Moreover, we obtain that $ 2^{-n(n-1)/2}  \left(\frac{2}{\pi}\right)^{-r_2} \le \co(J) \le \sqrt{|\Delta_F|}$.
\item[ii)] There is some $s \in \left(\prod_{\sigma} \mathbb{R}_{+}^*\right)^0$ and some $f \in I$ such that 
$$ D - D' + (f) = (O_F, s)$$
 and $$\|D - D'\|_{Pic} =\|s\|_{Pic} < \log{\mathcal{D}_F}.$$

\end{itemize} 
\end{proposition}

\begin{proof}
i) Since $D'$ is obtained from an LLL-reduction on $D$, there is an LLL-reduced basis $\{ b_1, ..., b_n\} $ of the lattice $uI$ associated to $D$ such that $b_1 = u f$ and $J =  f^{-1} I $ for some $f \in I$. As $f \in I$, the ideal $J$ contains $1$. Thus $J^{-1}$ is integral.
	
By Lemma \ref{lem:min}, there is a nonzero element $g \in J$ such that
$$ N(J)^{-1/n} |\sigma(g)| \leq \partial_F^{1/n} \text{ for all } \sigma.$$
Hence
\begin{equation} \label{eq1}
N(J)^{-1/n} \max_{\sigma} |\sigma(g)|\leq  \partial_F ^{1/n}.\hspace*{5cm}
\end{equation} 
We have 
$$\|u (f g) \| = \|(uf) g\|=  \|b_1 g\|\leq \max_{\sigma} |\sigma(g)|  \|b_1\|.$$
Furthermore, $\|b_1\| \leq 2^{(n-1)/2} \|u (f g) \| $ since $u (f g) \in u I$ and by the property of LLL-reduced bases {\cite[Section 10]{ref:1}}. As a consequence,
\begin{equation} \label{eq2}
\max_{\sigma} |\sigma(g)|  \ge \frac{\|u (f g) \|}{\|b_1\|} \ge 2^{-(n-1)/2}.
\end{equation} 
The inequalities \eqref{eq1} and \eqref{eq2} imply that
 $$ N(J^{-1})  \leq \frac{\partial_F}{(\max_{\sigma} |\sigma(g)| )^n } \le 2^{n(n-1)/2} \partial_F.$$

Therefore the first statement in i) is proved. Since $\co(J) = \sqrt{|\Delta_F|} N(J)$, the second statement in i) follows.

ii) 
The divisor $D = (I, u)$ is of degree $0$,  by Lemma \ref{lem:min}, there is a nonzero element $f'$ in $I$ such that 
$\|u f'\|  \leq \sqrt{n} \hspace*{0.1cm}\partial_F^{1/n}.$ 
Since $b_1$ is the first vector in an LLL-reduced basis of the lattice $u I$, we again have $\| b_1\| \leq 2^{(n-1)/2} \| u f'\| $ {\cite[Section 10]{ref:1}}. It follows that $ \|b_1\| \leq \sqrt{n}\hspace*{0.1cm} 2^{(n-1)/2} \partial_F^{1/n} = \mathcal{D}_F^{1/n}$.

 Let $ s = u |f| N(J)^{1/n} = |b_1| N(J)^{1/n}$. Then with the notation in i), we have that 
  $ D- D' + (f) = (O_F, s)$. 
In particular, $D-D'= (O_F, s) \in \Pic_F^0$. By Section \ref{sec:metric}, this leads to the following.  
$$\|D- D'\|_{Pic} =  \|s\|_{Pic}.$$ 

 Part i) shows that $J^{-1}$ is integral, so $N(J) \leq 1$. Therefore, the following inequality holds.
 $$s_{\sigma}\leq  \|s\|=  \|b_1\|  N(J)^{1/n} \leq \mathcal{D}_F^{1/n}  \text{ for all  } \sigma.$$
This leads to  $\log(s_{\sigma}) \le  \frac{1}{n}\log{\mathcal{D}_F}$ for all $\sigma$. Since $\sum_{\sigma}\deg(\sigma)\log{s_{\sigma}} =0 $, we can  easily prove the following {\cite[Lemma 7.5]{ref:4}}.
$$\| \log{s}\|^2 = \sum_{\sigma  }  \deg(\sigma)|\log{s_{\sigma}}|^2  \leq n(n-1) \left(\frac{1}{n}\log{\mathcal{D}_F} \right)^2 < \log^2{\mathcal{D}_F}.$$
Since $\|D - D'\|_{Pic} =\|s\|_{Pic} \leq \| \log{s}\|$, part ii) is proved.

\end{proof}

\begin{definition}
Let $W = (I,v)$ be an Arakelov divisor of degree $d$. We call $D=(I, u )$ with $u = e^{d/n} v$ the \textit{divisor translated from} $W$. 
\end{definition}
Note that if $D$ is translated from a divisor $W$ then $deg(D)=0$. In other words, the class of the divisor $D$ is in $\Pic_F^0$.

We prove the following corollary.

\begin{corollary}\label{pro:norm} 
Let $W = (I,v)$ be an Arakelov divisor of degree $d$ and let $D$ be the divisor translated from $W$. 
Assume that $D' = d(J)$ is a divisor obtained by an LLL-reduction on $D$. Then $L =  e^{-d/n}  N(J)^{-1/n}  s J$ is the lattice associated to $W$ for some $s \in \left(\prod_{\sigma} \mathbb{R}_{+}^*\right)^0$ and $\|s\|_{Pic} < \log{\mathcal{D}_F}$. 
\end{corollary}

\begin{proof}
 By Proposition \ref{pro:close}, there exists some $s \in \left(\prod_{\sigma} \mathbb{R}_{+}^*\right)^0$ for which $\|s\|_{Pic} < \log{\mathcal{D}_F}$ and $f \in I$ such that   $D +(f) = D' + (O_F,s) $ in $\Div_F^0$. Then so  $D  = D' + (O_F,s) = (J,  N(J)^{-1/n} s)$ in $\Pic_F^0$.\\ 
Since $W =D +(O_F, e^{-d/n} ) \in \Pic_F$, we have the following.
$$W = D + (O_F, e^{-d/n}) = \left(J, e^{-d/n}  N(J)^{-1/n} s \right) \in \Pic_F.$$
Thus, the lattice associated to $W$ is $L = e^{-d/n}  N(J)^{-1/n} s J$.
\end{proof}

\section{Computing The Function $h^0$ }\label{sec:4}
Let $W = (I,v)$ be an Arakelov divisor of degree $d$ and let $D$ be the divisor translated from $W$. Assume that a basis for the ideal lattice $I$ and the coordinates of the vector $v$ are known. We compute an approximate value of $h^0$ at the class of the Arakelov divisor $W \in \Pic_F$ with some given error $\delta$. We consider the case in which $v$ is a  long vector and the discriminant $\Delta_F$ is quite large since it is quite trivial to compute $h^0(W)$ in other cases.

We have that
 \begin{equation*}\label{eq:0}
 h^0(W)= \log{\sum_{f \in I}e^{-\pi \|f v\|^2 }}.
\end{equation*} 

We approximate the value of $h^0(W)$ with some small error. This can be done by summing up only the large terms, i.e., the terms
$e^{-\pi  \|f v\|^2 } $ for which  $ \|f v\|^2\leq M$ with some given $M>0$.
In case $v$ is a long vector, while collecting such short vectors $f$, it is quite easy to miss many of them.  Consequently, the obtained value of $h^0(W)$ may be smaller than the true value. 
Therefore, we will find some ``good" divisor  $D'$, that is obtained from an LLL-reduction on $D$ and has nice properties described in Section \ref{result}, then use it  for computing $h^0(W)$. 

Note that by Proposition \ref{pro:close}, for any given divisor $D$ of degree $0$, there exists a good divisor $D'$ close to $D$ in $\Pic_F^0$ in the sense that $\|D-D'\|_{Pic} < \log{\mathcal{D}_F}$.

We first describle the following algorithm that is similar to (cf.[\cite{ref:4}, Algorithm 10.4]). 
\begin{alg}\label{alg1}  
Given two Arakelov divisors $D_1=d(J_1)$ and $D_2= d(J_2)$ such that $N(J_1^{-1}) \le 2^{n(n-1)/2} \partial_F$ and $N(J_2^{-1}) \le 2^{n(n-1)/2} \partial_F$, compute a divisor $d(J)$ obtained from the LLL reduction on $D_1 + D_2 \in \Pic_F^0$ in polynomial time in  $\log{|\Delta_F|}$. \\
\textit{Description.} Since $N(J_1^{-1}) \le 2^{n(n-1)/2} \partial_F$ and $N(J_2^{-1}) \le 2^{n(n-1)/2} \partial_F$, the result $D_3=D_1 + D_2 =(J_1 J_2, N(J_1 J_2))$ can be computed in time polynomial in $\log{|\Delta_F|}$. Then one performs the LLL reduction on the divisor $D_3$. The resulting divisor $d(J)$ is then close to $D_1 + D_2$ by Proposition \ref{pro:close}. Since $N(J_1 J_2)^{-1} \le 2^{n(n-1)} \partial_F^2$, the running time of this second step is also polynomial in  $\log{|\Delta_F|}$. 
\end{alg}

Next, we explain how to compute efficiently a divisor $D'$ obtained from some LLL reduction close to a given divisor $D= (O_F, u )$ in $\Pic_F^0$. This process can be seen as performing repeatedly doubling and LLL-reduction to go from the origin $(O_F,1)$ to $D$. We apply the ``jump algorithm" [10, Algorithm 10.8] with a minor modification to adapt to our situation. Indeed, in the reduction step, instead of using a shortest vector, we use the first vector of an LLL-reduced basis of the lattices associated to Arakelov divisors. 

\begin{alg}\label{alg2}
Given a divisor $D= (O_F, u )$ of degree 0, compute a reduced Arakelov divisor whose image in $\Pic^0_F$ has distance less than $\log{\mathcal{D}_F}$ from $D$.\\
\textit{Description.}
Assume that $u=(e^{-w_{\sigma}})_{\sigma}$. Let $t\geq 0$ be the smallest integer for which $n\cdot2^{-t}\cdot|w_{\sigma}|<\log \partial_F $ for all $\sigma$. Then $z_{\sigma}=2^{-t}\cdot w_{\sigma}$ satisfies $n\cdot|z_{\sigma}|<\log \partial_F $ for all $\sigma$. Let $\omega =(e^{-z_{\sigma}})_{\sigma}$. Then $\omega^{2^t}=u'$. In other words, from the point $(O_F,\omega)$ we can reach to $D$ after $t$ times doubling.
 Denote by $W_i = (O_F, \omega^{2^i})$ for $ i = 1, 2, ..., t$.  
  We inductively compute Arakelov divisors $D'_i = d(J_i)$ obtained by LLL reduction for which
  \begin{equation}\label{D'}
  \|W_i - D'_i \|_{Pic} \le \log{\mathcal{D}_F} .
  \end{equation}  
  We compute $D'_{i+1}$ from $W_i$ by doubling and doing LLL reduction. More precisely, by induction, there exists some $\omega_i \in \left(\prod_{\sigma} \mathbb{R}_{+}^*\right)^0$ such that  $W_i = D'_i + (O_F,\omega_i) $ in $\Pic_F^0$ and
  $$\|W_i - D'_i \|_{Pic}  = \|\omega_i\|_{Pic} < \log{\mathcal{D}_F}.$$
  
   Since $W_i = D'_i + (O_F,\omega_i)= (J_i, v_i)$ where $v_i = \omega_i N(J_i)^{-1/n})$, we get $W_{i+1}= 2 W_i = (J_i^2,v_i^2) $. Let $D'_{i+1}$ be a divisor obtained by an LLL reduction on $W_{i+1}$. 
   Then there is an LLL-reduced basis $\{ b_1, ..., b_n\} $ of the lattice $L_{i+1}= v_i^2  J_i^2$ associated to $W_{i+1}$ such that $J_{i+1} =  b_1^{-1} L_{i+1} $. 
  Proposition \ref{pro:close} shows that 
  $W_{i+1} = D'_{i+1} + (O_F,\omega_{i+1})  \text{ in } \Pic_F^0$ for some $\omega_{i+1} \in \left(\prod_{\sigma} \mathbb{R}_{+}^*\right)^0$ such that 
  $$\|W_{i+1} -D'_{i+1}\|_{Pic} =\|\omega_{i+1}\|_{Pic} < \log{\mathcal{D}_F}.$$
Note that here $\omega_{i+1} = |b_1| N(J_{i+1})^{1/n}$. Thus, we can construct all divisor $D_i'$ satisfying \eqref{D'} for $i=1, 2, ..., t$.

Now let $s = \omega_t$ and let $D' = d(J_t)$. Then
$$D = W_t = D'_t + (O_F,\omega_{t})= D' + (O_F, s) \text{ in } \Pic_F^0$$
 for $s = \omega_t \in \left(\prod_{\sigma} \mathbb{R}_{+}^*\right)^0$ and $\|D-D'\|_{Pic}= \| s\|_{Pic} = \|\omega_t\|_{Pic} < \log{\mathcal{D}_F}$. This completes the description of the algorithm.
 
\end{alg}

With the notation of Proposition \ref{pro:close} and Corollary \ref{pro:norm}, let $L =  e^{-d/n}  N(J)^{-1/n}  s J$ be the lattice associated to $W$. Then
 \begin{equation}\label{eq:1}
 h^0(W)= \log{\sum_{f \in J}e^{-\pi e^{-2 d/n}  N(J)^{-2/n}  \|fs\|^2}} = \log{\sum_{g' \in  L}e^{-\pi \|g'\|^2}}.
\end{equation} 

Every vector of the lattice $L$ has the form $ e^{-d/n}  N(J)^{-1/n}  s f$ for some $f \in J$. 
As $s$ is short, $N(J)^{-1/n}$ is a small scalar and $J$ is a nice lattice (see Proposition \ref{pro:close}), we can easily compute an LLL-reduced basis of $L$ then find short vectors of the lattice $L$ more efficiently. Therefore, the value of $h^0(W)$ can be computed more exactly.
Computing $h^0(W)$ is done in 3 steps described in Section \ref{sec:jump}, \ref{sec:poi} and \ref{sec:short} in succession.
\subsection{Finding a good divisor $D'$ close to $D$}\label{sec:jump}
Assume that a basis for the ideal lattice $I$ and the coordinates of the vector $v$ are known. 
We will find a  divisor $D'$ that is obtained from some LLL reduction and with the property that
$D = D' + (O_F, s)$ in $\Pic_F^0$ for some $s \in \left(\prod_{\sigma} \mathbb{R}_{+}^*\right)^0$ and 
$\|D -D'\|_{Pic} < 3 \log{\mathcal{D}_F}$.

Let $D_1=(I, N(I)^{-1/n})$ and to $D_2= (O_F, N(I)^{1/n} u)$. Then $D_1$, $D_2$ have degree $0$ and $D= D_1 +D_2$. We compute divisors $D_1'$ (see \ref{st1}) and $D_2'$ (see \ref{st2}) obtained from some LLL reduction so that  $\|D_1 -D_1'\|_{Pic} < \log{\mathcal{D}_F}$ and  $\|D_2 -D_2'\|_{Pic} < \log{\mathcal{D}_F}$. Then we find a divisor $D'=d(J)$ close $D_1' + D_2'$ in $\Pic_F^0$ (see \ref{st3}).
This process is described as follows.
 
 \subsubsection{Computing $D_1'$} \label{st1}
 We compute a divisor $D_1'$ close to $D_1$ in $\Pic_F^0$ in the sense that its distance to $D_1$ is at most $ \log \mathcal{D}_F$. 
  This can be done easily by performing the LLL reduction on $D_1$. By Proposition \ref{pro:close}, there is some $s_1 \in \left(\prod_{\sigma} \mathbb{R}_{+}^*\right)^0$ so that  $D_1 -D_1'=(O_F,s_1)$ in $\Pic_F^0$  and
  $\|D_1 -D_1'\|_{Pic} =\|s_1\|_{Pic} < \log{\mathcal{D}_F}.$ 
  
 \subsubsection{Computing $D_2'$}\label{st2} Use Algorithm \ref{alg2} to compute a reduced Arakelov divisor$D'_2$  whose image in $\Pic^0_F$ has distance less than $\log{\mathcal{D}_F}$ from $D_2= (O_F, u' )$ with $u'=N(I)^{1/n} u$. We obtain that 
 $$D_2 = D'_2 +  (O_F, s_2) \text{ in } \Pic_F^0$$
  for $s_2 \in \left(\prod_{\sigma} \mathbb{R}_{+}^*\right)^0$ and $\|D_2-D'_2\|_{Pic}= \| s_2\|_{Pic} < \log{\mathcal{D}_F}$.

 \subsubsection{Computing $D'$}\label{st3}  Adding divisors $D_1'$ and $D_2'$ as described in Algorithm \ref{alg1}, we then compute a divisor $D'=d(J)$ close to $D_1' + D_2'$ in $\Pic_F^0$.  
 Indeed, by performing LLL reduction on the divisor $D_1' + D_2'$ we obtain $D'=d(J)$ and  $(D_1' + D_2')-D'= (O_F, s_3)$ for some $s_3 \in \left(\prod_{\sigma} \mathbb{R}_{+}^*\right)^0$ and 
 $\|(D_1' + D_2')-D'\|_{Pic}= \| s_3\|_{Pic} < \log{\mathcal{D}_F}$.
 
 Let $s=s_1 \cdot s_2 \cdot s_3 \in \left(\prod_{\sigma} \mathbb{R}_{+}^*\right)^0$. Then
  $$D- D' = (D_1 -D_1') + (D_2 -D_2') -((D_1' + D_2')-D') = (O_F, s) \in \Pic_F^0.$$
  Thus, 
  $\|D- D'\|_{Pic} = \|s\|_{Pic} \le \|s_1\|_{Pic} + \|s_2\|_{Pic} + \|s_3\|_{Pic} \le  3 \log \mathcal{D}_F.$

  \subsection{Applying Poisson summation for the lattice $L$} \label{sec:poi}
 Since $D- D' =  (O_F, s) \in \Pic_F^0$, it follows that $D=\left(J,  N(J)^{-1/n}  s \right)$. 
  By translating  $D$ to the divisor $W$, we obtain that  
       $W= D+ (O_F,e^{-d/n} ) = \left(J, e^{-d/n}  N(J)^{-1/n}  s \right).$
       
  Let $L = e^{-d/n}  N(J)^{-1/n}  s J$ be the lattice associated to $D$. Assume that $L$ has an LLL-reduced basis $\{ \mathsf{b}_1, ..., \mathsf{b}_n\}$. 
   Let $\{ \mathsf{b}_1^*, ..., \mathsf{b}_n^*\}$ denote the Gram-Schmidt orthogonalization of this basis and
   $$\mu_{i, j} = \frac{\langle \mathsf{b}_i, \mathsf{b}_j^*\rangle}{\| \mathsf{b}_j^*\|^2} \text{ for all } 2 \le i \le n \text{ and } 1\le j \le i-1.$$ 
  
   Then any element $\mathsf{z} \in L$ can be written uniquely as $\mathsf{z} = \sum_{i=1}^n x_i \mathsf{b}_i $ with the coefficients $x_i \in  \mathbb{Z}$ for all $i =1, 2, ..., n$. Similar to Lemma \ref{poi3}, we can write $\|\mathsf{z}\|^2$ as below.
  \begin{equation}\label{eq:3a}
       \|\mathsf{z}\|^2 = \sum_{i=1}^{n} A_{i,i} \left(x_i + \sum_{j=1}^{i-1} A_{i, j} x_j\right)^2
     =q(x_1, x_2, ..., x_n).
     \end{equation}
       where $A_{i,i} =  \|\mathsf{b}_i^*\|^2 $ and $A_{i,j}=\mu_{i, j}$.   
         Therefore,
          \begin{equation}\label{eq:3}
           h^0(W)=\log{\sum_{\mathsf{z} \in L}e^{-\pi \|\mathsf{z}\|^2}}
         =\log{\sum_{x_i \in \mathbb{Z}}e^{-\pi q(x_1, x_2, ..., x_n)}}.
         \end{equation}

  \begin{remark}\label{re:Aii}
   We only catch the vectors $\mathsf{z}$ in $L$ for which $ \|\mathsf{z}\|^2 \leq M$. In other words, we only compute vectors $\mathsf{x}=(x_1, x_2, ..., x_n) \in \mathbb{Z}^n$ satisfying $q(\mathsf{x}) \leq M $. 
   The Fincke--Pohst algorithm {\cite[Algorithm 2.12]{ref:40}} or an LLL reduced basis of $L$ {\cite[Section 12]{ref:1}} can be used to find the list $\mathcal{L}$ of these vectors $\mathsf{x}$.
    
    An approximate value of $h^0$ is obtained by summing up only the terms $e^{-\pi q(\mathsf{x})} $ for which $\mathsf{x} \in \mathcal{L}$ as below.
    \begin{equation}\label{eq:4}
    h^0(W) \approx  \log{\sum_{\mathsf{x} \in \mathcal{L}}e^{-\pi q(\mathsf{x})}}.
    \end{equation}   
    
  The lattice $L$ has covolume  $\co(L)=\prod_{i=1}^{n} \|\mathsf{b}_i^*\|=\sqrt{\prod_{i=1}^{n}A_{ii}}$. The list $\mathcal{L}$ can have at most $\alpha(n) \frac{M^{n/2}}{\co(L)}$ vectors. Here $\alpha(n)$ is a function depending only on $n$. See Algorithm 2.12 in \cite{ref:40} and Section 12 in \cite{ref:1} for the explanation.  Therefore, in order to reduce the number of vectors in the list  $\mathcal{L}$, we can ``make" $\co(L)$ larger by using the Poisson summation formula as follows.
     \end{remark}

  Assume that $  \|\mathsf{b}_1^*\| =\|\mathsf{b}_1\| < 1$.
 Let $k$ be the largest index such that $ A_{i,i} <1 \text{ for all } i \leq k$. Denote by $$L_1 = \oplus_{i=1}^k \mathbb{Z} \cdot \mathsf{b}_i \hspace*{1cm} \text{and } \hspace*{1cm} L_2 = \oplus_{j=k+1}^n \mathbb{Z} \cdot \mathsf{b}_j.$$ 
 Then $L = L_1 \oplus L_2$. Since $\{\mathsf{b}_1, ..., \mathsf{b}_n\}$ is LLL-reduced, the vectors $\mathsf{b}_1, ..., \mathsf{b}_k$ form an LLL-reduced basis for $L_1$ (see \cite{ref:1}).

 \begin{remark}\label{dualbasic}
 Let $B$ be the matrix of which columns are vectors $\mathsf{b}_1, ..., \mathsf{b}_k$ and let $G= B^t \cdot B$. Then the columns of $B \cdot G^{-1}$ form a basis for the dual lattice $L_1^{\vee}$ of $L_1$.
 \end{remark}

 Now we apply the Poisson summation formula for $L$. See Lemma \ref{poi1}. Let $\gamma$ be the covolume of the lattice $L_1$ inside $V= L_1 \otimes \mathbb{R}$. Then $\gamma =\prod_{i=1}^k \|\mathsf{b}_{i}^*\|$ and 
  \begin{equation}\label{e1}
    h^0(W) = \log{\left( \frac{1}{\gamma}\sum_{\mathsf{b} \in L_2}e^{-\pi \|\mathsf{b}-\pi(\mathsf{b})\|^2}\sum_{\mathsf{a} \in L_1^{\vee}} e^{-\pi \|\mathsf{a}\|^2 -2 \pi i \langle \mathsf{a} , \pi(\mathsf{b})\rangle }\right)}.
  \end{equation}

  Assume that $L_1^{\vee}$ has a basis $\mathsf{c}_1,\cdots , \mathsf{c}_k$ that is computed by Remark \ref{dualbasic}. Denote by $\{ \mathsf{c}_1^*, \cdots, \mathsf{c}_k^*\}$ the Gram-Schmidt orthogonalization of the basis $\{ \mathsf{c}_1, \cdots, \mathsf{c}_k\}$ and 
  $$C_{i, i} = \|\mathsf{c}_i^*\|^2, \hspace*{0,5cm}C_{i,j} = \frac{\langle \mathsf{c}_i, \mathsf{c}_j^*\rangle}{\| \mathsf{c}_j^*\|^2} \text{ for all } 1 \le i \le k \text{ and } 1\le j \le i-1.$$
  
  Now let $\mathsf{a}= \sum_{i=1}^{k} x_i \mathsf{c}_i \in L_1^{\vee}$ where $x_i \in \mathbb{Z}$ for all $i=1, 2, \cdots, k$ and $\mathsf{b} = \sum_{j=k+1}^{n} x_j \mathsf{b}_j\in L_2$ where $x_j \in \mathbb{Z}$ for all $j=k+1, \cdots, n$. 
 
 \begin{lemma}\label{poi3}
 We have 
 $$C_{i,i} = \|\mathsf{c}_i^*\|^2 = \frac{1}{\|\mathsf{b}_i^*\|^2} \text{ for all } 1 \le i \le k $$
 and moreover
 \begin{equation}\label{e3}
   \|\mathsf{a}\|^2 = \sum_{i=1}^{k} C_{i,i} \left(x_i + \sum_{r=i+1}^{k} C_{r,i} x_r \right)^2.
  \end{equation}
 
  \end{lemma}
  \begin{proof}
  This is easily proved by using Remark \ref{dualbasic} and properties of the Gram-Schmidt orthogonal basis $\{ \mathsf{c}_1^*, \cdots, \mathsf{c}_k^*\}$.
  \end{proof}

Recall that $A_{j,j} = \|\mathsf{b}_j^*\|^2$ and $A_{t,j} = \mu_{t,j}$ for all $j= k+1, \cdots, n$ and $t> j$. We have the lemmas below.

  \begin{lemma}\label{poi4} 
  We have
   \begin{equation}
     \langle \mathsf{a}, \pi(\mathsf{b})\rangle=  \sum_{l=1}^{k}  \sum_{j=k+1}^{n} \sum_{i=1}^{k} A_{j, i} \langle \mathsf{c}_l,\mathsf{b}_i^* \rangle x_l x_j.
    \end{equation}
   \end{lemma}
   
   \begin{proof}
    Because $\pi(\mathsf{b})$ is the orthogonal projection of $\mathsf{b}$ on $V$ that has an orthogonal basis $\mathsf{b}_1^*, \cdots, \mathsf{b}_k^*$, we obtain that
    \begin{equation}\label{equ1}
    \pi(\mathsf{b}) = \sum_{i=1}^{k} \frac{\langle \mathsf{b}, \mathsf{b}_i^* \rangle}{\|\mathsf{b}_i^*\|^2} \mathsf{b}_i^* = \sum_{i=1}^{k} \sum_{j=k+1}^{n} x_j A_{j,i} \mathsf{b}_i^* .
    \end{equation}
    Then the result is implied by taking scalar product of $ \pi(\mathsf{b})$ with $\mathsf{a}= \sum_{l=1}^{k} x_l \mathsf{c}_l$.
   \end{proof}

 \begin{lemma}\label{poi2} 
 We have
 $$\|\mathsf{b} -\pi(\mathsf{b})\|^2 = \sum_{j=k+1}^{n} A_{j,j} \left(x_j + \sum_{t=j+1}^{n} A_{t, j} x_t\right)^2.$$
 \end{lemma}

 \begin{proof}
 Since $b_j = \mathsf{b}_j^* + \sum_{t=1}^{j-1} \mu_{j,t} \mathsf{b}_t^*$ for all $j \ge 2$, the vector  $\mathsf{b}$ therefore can be rewritten as
 
 \begin{equation}\label{e4} 
    \mathsf{b} = \sum_{j=k+1}^{n} \left(x_j +\sum_{t=j+1}^{n} \mu_{t, j} x_t \right) \mathsf{b}_j^*  + \sum_{i=1}^{k} \sum_{j=k+1}^{n} x_j \mu_{j,i} \mathsf{b}_i^*.
    \end{equation}
By using equalities \eqref{e3} and \eqref{e4}, the result is obtained since the vectors $\mathsf{b}_{k+1}^*, ..., \mathsf{b}_n^*$ are pairwise orthogonal.
 \end{proof}

Lemma \ref{poi3}, \ref{poi4}, \ref{poi2} and \eqref{e1} lead to 
  \begin{equation}\label{for:Poik} 
  h^0(W)=\log{\left(\frac{1}{\gamma}\sum_{x_i \in \mathbb{Z}}e^{-\pi Q(x_1, x_2, ..., x_n)  } \right)} .
  \end{equation} 
   where $Q(x_1, \cdots, x_n) = Q_1(x_1,..., x_n)+ 2Q_2(x_1, \cdots,x_n)i$ with
  \begin{multline}\label{for:Poi} 
     Q_1(x_1,..., x_n) = \|\mathsf{b}-\pi(\mathsf{b})\|^2 + \|\mathsf{a}\|^2\\
       = \sum_{j=k+1}^{n} A_{j,j} \left(x_j + \sum_{t=j+1}^{n} A_{t, j} x_t\right)^2+ \sum_{i=1}^{k} C_{i,i} \left(x_i + \sum_{r=i+1}^{k} C_{r,i} x_r \right)^2
     \end{multline}
  and 
  $$ Q_2(x_1, \cdots,x_n) = \langle \mathsf{a} , \pi(\mathsf{b}) \rangle = \sum_{l=1}^{k}  \sum_{j=k+1}^{n} \sum_{i=1}^{k} A_{j, i} \langle \mathsf{c}_l,\mathsf{b}_i^* \rangle x_l x_j.$$

\subsection{Finding the short vectors of the lattice associated to $D$}\label{sec:short}    
  An approximation of $h^0(W)$ is obtained by summing up the terms $e^{-\pi Q(\mathsf{x}) }$ such that $ Q_1(\mathsf{x}) \leq M $. By using the Fincke--Pohst algorithm, we can find the list $\mathcal{L}_1$ containing all vectors $\mathsf{x}=(x_1, \cdots, x_n) \in \mathbb{Z}^n$ such that $ Q_1(\mathsf{x}) \leq M $. See Algorithm 2.12 in \cite{ref:40}. 
  Then an approximate value of $h^0(W)$  is obtained as follows.
 \begin{equation}\label{eq:8}
  h^0(W) \approx  \log{\left(\frac{1}{\gamma}  \sum_{\mathsf{x} \in \mathcal{L}_1}e^{-\pi Q(x_1, ..., x_n)} \right)}.
  \end{equation}

  \begin{remark} \label{expla}
  Let $\mathcal{L}_1 = \{ \mathsf{x} \in \mathbb{Z}^n: Q_1(\mathsf{x}) \leq M\}$ and $\mathcal{L}=\{ \mathsf{x} \in \mathbb{Z}^n: q(\mathsf{x}) \leq M\}$. Let $L'$ be the lattice associated to the quadratic form $Q_1(\mathsf{x})$. 
  Then
  \begin{equation}\label{eq:col}
\co(L')^2 =   \prod_{i=1}^{k} C_{i,i} \prod_{j=k+1}^{n} A_{j,j} = \frac{1}{\prod_{i=1}^{k} A_{i,i} }\prod_{j=k+1}^{n} A_{j,j} = \frac{1}{(\prod_{i=1}^{k} A_{i,i})^2}  \co(L)^2.
  \end{equation}
  
    We have that 
     $\#\mathcal{L} \le \alpha(n) \frac{M^{n/2}}{\co(L)}$ and   $\#\mathcal{L}_1 \le \alpha(n) \frac{M^{n/2}}{\co(L')}$ (see Algorithm 2.12 in \cite{ref:40} and Section 12 in \cite{ref:1}). 
  Since $A_{i,i}<1$ for all $i \le k$, it follows that $\co(L') \ge \co(L)$. From this inequality and \eqref{eq:col}, we usually obtain that $\#\mathcal{L}_1 \le \#\mathcal{L}$.
  
   Thus, the list $\mathcal{L}_1$ usually  contains less vectors than the list $\mathcal{L}$ does. In addition,  $\frac{1}{\gamma}=\frac{1}{\sqrt{\prod_{i=1}^k A_{i,i}}} > 1$. At the result, the sum in \eqref{eq:8} converges better than in \eqref{eq:4}. 
  Hence, we can compute \eqref{eq:8} by only summing a small number of terms.
  
 Note that the function \texttt{qfminim} in \texttt{pari-gp} that uses the Fincke--Pohst algorithm, can be used to find all nonzero vectors (up to a sign) with length bounded by $M$ of a given lattice.
 Another method uses an LLL reduced basis of the lattice $L'$; see Section 12 in \cite{ref:1}. For a fixed lattice, the complexity of both methods is in polynomial time in $M$ (see Section \ref{sec:runtime}). 

By the proofs of Lemma \ref{error}, \ref{Q1} and Proposition \ref{esterr}, to approximate $h^0(W)$ with an error $\delta$, we can choose $M \approx \frac{1}{\pi -1 } \left( \log(1/\delta) + (n+1) \log 3 + (n(n+1)/2-1) \log 2 \right)$.
 
  \end{remark}

The  algorithm below computes an approximate value of $h^0(W)$ with a given error $\delta$.\\
\textbf{Input:}     
\begin{itemize}
     \item  A basis for the lattice $I$.
     \item  The coordinates of $v$.
     \item  An error $\delta$.
\end{itemize}
\textbf{Output:} An approximate value of $h^0(W)$ with error $\delta$.\\

\fbox{\begin{minipage}{0.95\textwidth}
 \begin{alg}\label{alg}

\begin{arabiclist}
 \item Find a divisor $D'$ that is close to $D$ in $\Pic^0_F$ as described in Section \ref{sec:jump}.

 \item Apply Poisson summation formula.
     \begin{alphlist}[(a)]
         \item Find an LLL-reduced basis
            $\{ \mathsf{b}_1, ..., \mathsf{b}_n\}$ of $L'$.
         \item Compute $\{ \mathsf{b}_1^*, ..., \mathsf{b}_n^*\}$ and 
           $A_{i,i} =  \|\mathsf{b}_i^*\|^2 $ and 
           $A_{i,j} = \frac{\langle \mathsf{b}_i, \mathsf{b}_j^*\rangle}{\| \mathsf{b}_j^*\|^2} \text{ for all } 2 \le i \le n \text{ and } 1\le j \le i-1.$ 
      
        \item If $\|\mathsf{b}_1\| \ge 1$, then put  $Q(x_1, ..., x_n)=Q_1(x_1, ..., x_n)=q(x_1, ..., x_n)$ (see \eqref{eq:3a}) and $Q_2(x_1, ..., x_n)=0$, $L'=L$ and $\gamma=1$.       
        
         If $\|\mathsf{b}_1\| < 1$, then let $k$ be the largest index such that $ \|\mathsf{b}_j^*\| <1 \text{ for all } j \leq   k$.  Denote by $B$ the matrix of which columns are vectors $\mathsf{b}_1, ..., \mathsf{b}_k$. 
             \begin{romanlist}
                \item[(i)] Compute $G= B^t \cdot B$ and  $C= B \cdot G^{-1}$.
                \item[(ii)] Let $\mathsf{c}_1,\cdots , \mathsf{c}_k$ be the columns of $C$. Compute $\{ \mathsf{c}_1^*, \cdots, \mathsf{c}_k^*\}$ and 
                $C_{i, i} = \|\mathsf{c}_i^*\|$ and $C_{i,j} = \frac{\langle \mathsf{c}_i, \mathsf{c}_j^*\rangle}{\| \mathsf{c}_j^*\|^2} \text{ for all } 1 \le i \le k \text{ and } 1\le j \le i-1.$
                \item[(iii)] Compute $\langle \mathsf{c}_l,\mathsf{b}_i^* \rangle$ for all $l=1, ...,k$ and $i=1,...,k$.
                \item[(iv)] Denote $Q(x_1,..., x_n)$, $Q_1(x_1,..., x_n)$ and $Q_2(x_1,..., x_n)$ as in \eqref{for:Poi} and let $L'$ be the lattice associated to $Q_1(\mathsf{x})$ and $\gamma=\sqrt{\prod_{i=1}^k A_{i,i}}$.
            \end{romanlist} 

     \end{alphlist}

 \item  Find the short vectors of the lattice $L'$.
   \begin{alphlist}[(a)]
   \item  Compute  $M = \frac{1}{\pi -1 } \left( \log(1/\delta) + (n+1) \log 3 + (n(n+1)/2-1) \log 2 \right).$
   \item  Find the list $\mathcal{L}_1=\{\mathsf{x}=(x_1, \cdots, x_n) \in \mathbb{Z}^n:  Q_1(\mathsf{x}) \leq M \}$ and approximate $h^0(W)$ as \eqref{eq:8}.
   \end{alphlist}

\end{arabiclist}
     
\end{alg}  
  
\end{minipage}}

\section{The Error and Running Time of The Algorithm}\label{sec:5}

\subsection{Bound for the error in Algorithm \ref{alg} }\label{sec:error}                
To find a bound for the error in approximating the value of $h^0(W)$ in Algorithm \ref{alg}, we use the idea of  {\cite[Section 4]{ref:21}} as below.  

\begin{lemma}\label{error}
	Let $L' $ be a lattice of rank $n$. Assume that the length of shortest vector of the lattice $L'$ is $\lambda $ and $M\geq \max\{\lambda^2, \frac{n}{2} \log{\frac{n}{2}}\} $. \\
	Let 
	$$S'= \sum_{\substack{ \mathsf{a} \in L'   \\ \|\mathsf{a}\|^2 > M }}e^{-\pi \|\mathsf{a}\|^2}.$$
Then $S'=O(\lambda^{-n} e^{-(\pi -1) M})$.
In particular, the bound for $S'$  goes to zero when $M$ tends to  infinity.                          
		               
\end{lemma}

\begin{proof} 	
Let $B_t = \{\mathsf{a} \in L': M \leq \|\mathsf{a}\|^2 \leq t \} \text{ for each } t> M$.
	The balls with centers $\mathsf{a} \in B_t$ and radius $\lambda/2$ are disjoint. Their union is contained in the (hyper) annular disk 
		$$\{ \mathsf{z} \in F_{\mathbb{R}}: \sqrt{M} - \lambda/2 \leq \|\mathsf{z}\| \leq  \sqrt{t} +  \lambda/2\}.$$
		Consequently, the following is implied.
		$$ \left(\frac{\lambda}{2}\right)^n \#B_t \leq  \left(\sqrt{t} +\frac{\lambda}{2}\right)^n - \left(\sqrt{M} -\frac{\lambda}{2}\right)^n.$$
		This leads to
		$$\#B_t \leq \left(1+ \frac{2\sqrt{t}}{\lambda} \right)^n -\left( \frac{2\sqrt{M}}{\lambda} -1\right)^n< \left( \frac{3\sqrt{t}}{\lambda} \right)^n -\left( \frac{2\sqrt{M}}{\lambda} -1\right)^n.$$
		The second inequality is since $t> M \geq \lambda^2$. 
	    Using this inequality, we get
		
		$$S' = \sum_{\substack{ \mathsf{a} \in L' \\ \|\mathsf{a}\|^2 > M }}\int_{\|\mathsf{a}\|^2}^{\infty} \! \pi e^{- \pi t}\, \mathrm{d}t  \leq   \pi  \int_{M}^{\infty} \! \# B_t e^{- \pi t}\, \mathrm{d}t $$
		$$\leq \pi  \int_{M}^{\infty} \! \left(\frac{3\sqrt{t}}{\lambda} \right)^n e^{- \pi t}\, \mathrm{d}t- \pi \int_{M}^{\infty} \!\left( \frac{2\sqrt{M}}{\lambda} -1\right)^n  e^{- \pi t}\, \mathrm{d}t.$$
		Since $M\geq \frac{n}{2} \log{\frac{n}{2}}$, we have $\left(\frac{3\sqrt{t}}{\lambda} \right)^n < \left(\frac{3}{\lambda} \right)^n e^{t} $. This implies that the first integral is at most 
		$\frac{1}{\pi -1} \left(\frac{3}{\lambda}\right)^n \hspace*{0.1cm} e^{-(\pi -1) M}$. The second one is equal to $\frac{1}{\pi} \left( \frac{2\sqrt{M}}{\lambda}-1\right)^n e^{-\pi M}$. Hence
		$$S' \leq \frac{\pi}{\pi -1} \left(\frac{3}{\lambda}\right)^n   e^{-(\pi -1) M} - \left( \frac{2\sqrt{M}}{\lambda}-1 \right)^n e^{-\pi M}.$$
		Thus, the lemma is proved.
	
\end{proof}

\begin{lemma}\label{Q1}
Let $L'$ be the lattice in $\mathbb{R}^n$ associated to the definite positive quadratic form $Q_1(\mathsf{x})$ in \eqref{for:Poi} of Section \ref{sec:poi}. 
Then the shortest vector of the lattice $L'$ has length 
$\lambda \ge 2^{(-n+1)/2}.$
	         	
\end{lemma}   

\begin{proof}

If $ A_{1,1}=\|b_1\|^2 \ge 1$, then by the property of LLL-reduced bases, the length of the shortest vector of $L'$ is at least $2^{(-n+1)/2} \|b_1\|$ {\cite[Section 10]{ref:1}}.  In other words,  $\lambda \ge 2^{(-n+1)/2}$. 

If $ A_{1,1} < 1$, then $\frac{1}{A_{i,i}} >1$ for all $i \leq k$ and $ A_{k+1, k+1} \geq 1$ since $k$ is the largest index such that $ A_{i,i} <1 \text{ for all } i \leq k$.  On the other hand, we have $A_{i,i} = \|\mathsf{b}_i^*\|^2$ for all $i = 1, 2, ..., n$. So, if   $k+2 \leq j \leq n$ then  $A_{j,j} = \|\mathsf{b}_j^*\|^2 \geq 2^{-(j-k-1)} \|\mathsf{b}_{k+1}^*\|^2 =  2^{-(j-k-1)} A_{k+1,k+1}  \geq 2^{-n+1}$  {\cite[Section 10]{ref:1}}.
Thus, all the coefficients $C_{i,i}=\frac{1}{A_{i,i}}$  with $i\leq k$ and $A_{jj}$ with $j\geq k+1$ are at least $2^{-n+1}$. As the result, 
	$$  Q_1(\mathsf{x}) \geq \min{\left\{ C_{i,i}, A_{j,j}: 1 \leq i \leq k \text{ and } k+1 \le j \le n \right\} }\geq 2^{-n+1}.$$
The result now follows since $\lambda^2 = \min\{Q_1(\mathsf{x}): \mathsf{x} \in \mathbb{Z}\}$. 
\end{proof}

\begin{proposition}\label{esterr}
Let $\delta$ be the error in approximating $h^0(W)$ described in Algorithm \ref{alg}. Then for fix degree of the number field, and for $M\geq \max\{\lambda^2, \frac{n}{2} \log{\frac{n}{2}}\}$,  we have $\delta = O( e^{-(\pi -1) M})$.
\end{proposition}  
\begin{proof}
	Now let 
	$$S= \sum_{\substack{ \mathsf{x} \in \mathbb{Z}^n  \\ Q_1(\mathsf{x}) \leq M }}e^{-\pi Q(\mathsf{x}) } \qquad\text{ and } \qquad S_0=\sum_{\mathsf{x} \in \mathbb{Z}^n}e^{-\pi Q(\mathsf{x}) }.  $$
	Then $h^0(W) = \log(\frac{1}{\gamma}S_0)$ and we approximate it by $\log(\frac{1}{\gamma}S)$. The error in approximating $h^0(W)$ is 
	$$\delta=\left| \log\left(\frac{1}{\gamma}S_0\right) - \log\left(\frac{1}{\gamma}S\right) \right|  = \left| \log{\frac{S_0}{S}} \right|.$$	
	Furthermore,
	$$| S_0 -S| = \left|  \sum_{\substack{ \mathsf{x} \in \mathbb{Z}^n  \\ Q_1(\mathsf{x}) > M }}e^{-\pi Q(\mathsf{x}) }  \right|
	 \leq \sum_{\substack{ \mathsf{\mathsf{x}} \in \mathbb{Z}^n  \\ Q_1(\mathsf{x}) > M }}|e^{-\pi Q(\mathsf{x}) } |  = \sum_{\substack{ \mathsf{x} \in \mathbb{Z}^n  \\ Q_1(\mathsf{x}) > M }}e^{-\pi Q_1(\mathsf{x}) } = \sum_{\substack{ \mathsf{a} \in L'   \\ \|\mathsf{a} \|^2 > M }}e^{-\pi \|\mathsf{a} \|^2 }.$$
	Since $M\geq \max\{\lambda^2, \frac{n}{2} \log{\frac{n}{2}}\} \ge  2^{(1-n)}$,  Lemma  \ref{error} and \ref{Q1} show that
	$|S_0-S| \le S'= O( e^{-(\pi -1) M}).$

	Since $S_0 > S >1$, it follows that 
	$$\delta=\left|\log{\frac{S_0}{S}}\right| \le |S-S_0| < O( e^{-(\pi -1) M}).$$
	
\end{proof}

\subsection{Run time of Algorithm \ref{alg}}\label{sec:runtime}  
	We prove the proposition below.

\begin{prop}
Let $0< \delta <1$. Assume that the given basis of the ideal lattice $I$ and the vector $v$ have size bounded by $|\Delta_F|^{O(1)}$. If the degree $n$ of the number field is fixed, then Algorithm \ref{alg} with an error $\delta$ runs in time  in $ \pol(\log(1/ \delta) \cdot \log{|\Delta_F|})$.
\end{prop}

\begin{proof} 
The basis $\{b_1, ..., b_n\}$ of the ideal lattice $I$ and vector $v$ have size at most $|\Delta_F|^{O(1)}$. Therefore,  Step 1--finding a good divisor $D'$ close to $D$ by using the ``jump algorithm"--runs in time polynomial in $\log{|\Delta_F|} $ {\cite[Algorithm 10.8]{ref:4}}. In addition, the entries of the matrix $B$ bounded by $|\Delta_F|^{O(1)}$ since they  are coordinates of $\{b_1, ..., b_n\}$. Thus, each Step 2a), 2b) and 2c) and hence Step 2 can be done in  polynomial time in $\log{|\Delta_F|}$.

The list $\mathcal{L}_1=\{\mathsf{x}=(x_1, \cdots, x_n) \in \mathbb{Z}^n:  Q_1(\mathsf{x}) \leq M \}$ can be computed by the Fincke--Pohst algorithm (see Algorithm 2.12 in \cite{ref:40}). If the degree $n$ of the number field is fixed, then the complexity of this algorithm is at most $1/ \co(L') \cdot \pol(M)$. See Section 3 in \cite{ref:40} for more details.
The covolume of  the lattice $L'$ is $\co(L') = \prod_{i=1,k}\|\mathsf{c}_i^*\| \cdot \prod_{j=k+1,n}\|\mathsf{b}_j^*\|$.  
By a similar argument in the proof of Lemma \ref{Q1}, one can show that $\|\mathsf{c}_i^*\|= C_{i,i} >1$ and $\|\mathsf{b}_j^*\| = \sqrt{A_{j,j}} \ge 2^{-(j-k-1)/2}$ for all $1 \le i \le k$ and $k+1 \leq j \leq n$. Consequently, $\co(L') \ge 2^{-n(n+2)/4}$.	 
Therefore, the complexity of Step 3 is bounded by $\pol(M)$. Proposition \ref{esterr} says that $M$ is bounded by $O(\log(1/\delta))$. As the result, Step 3 can be done in time $\pol(\log(1/\delta))$.	 
	 
Overall, the algorithm runs in time in $\pol(\log(1/ \delta) \cdot \log{|\Delta_F|})$ for fixed degree $n$ of the number field.		
		
\end{proof}

\section{Some Numerical Examples} 
We compute the value of $h^0$ for real quadratic fields and number fields with unit group of rank $2$. In the examples below, we pick an irreducible polynomial $P$ of large discriminant and compute $h^0$ for the number field $F$ defined by $P$. The algorithm works well without requiring the units of $F$.
Here \textit{pari - gp} is used to compute approximate values of $h^0$ and \textit{Mathematica} is used to plot it.

Since the symmetry induced by Riemann-Roch (see Proposition 1 in \cite{ref:3}), the graphs of $h^0$ on the cosets of $\Pic_F$ are similar.  See Example 1, 2 and 3 in \cite{ref:3}. In the following examples, we compute $h^0$ on the coset $\Pic_F^{(d)}$ of class of divisors of degree $d=1/2 \log \Delta_F$. 
 
 \begin{example}\label{ex:1}
  Let $\Delta_F= 10^{80}+129$ and $P = X^2- \Delta_F$ be the polynomial defining $F=\mathbb{Q}(\sqrt{\Delta_F})$. Then $F$ is a real quadratic field with the discriminant $\Delta_F$ and with  two real infinite  primes 
  $\sigma_1:  F \longrightarrow \mathbb{R}^2$ sends $\sqrt{\Delta_F}$ to itself and $\sigma_2:  F \longrightarrow \mathbb{R}^2$ sends $\sqrt{\Delta_F}$ to $-\sqrt{\Delta_F}$. 
  
  The class number of $F$ is 1 and the group $\Pic_F$ is isomorphic to a cylinder. For every $d \in \mathbb{R}$, the coset $\Pic_F^{(d)}$ of classes of degree $d$ is a circle whose circumference is equal to the regulator $R_F$ of $F$.
  
 We have $\left(\mathbb{R}^2\right)^0 = \{  (x_1,x_2): x_1, x_2 \in \mathbb{R} \text{ and } x_1 + x_2 =0 \} = \{  (-x_1,x_1): x_1 \in \mathbb{R}  \}$  is the bisector of the second quadrant of the axes. It is a 1-dimensional subspace of $\mathbb{R}^2$ with an orthonormal basis $\mathsf{e} = (-\frac{1}{\sqrt{2}}, \frac{1}{\sqrt{2}})$. The connected component of identity of $\Pic_F^0$ is a circle $\To = \left(\mathbb{R}^2\right)^0\slash \Lambda$ where $\Lambda$ is the lattice $\{(\log{|\sigma_1(f)|},\log{|\sigma_2(f)|)}: f \in O_F^* \}$ (that is unknown).

Denote by  $w=10^{20} \cdot \mathsf{e}  =  (-\frac{10^{20}}{\sqrt{2}}, \frac{10^{20}}{\sqrt{2}})= (w_1, w_2)$.  Let $W = (O_F,v) \in \Pic_F^{(d)}$ where $v= |\Delta_F|^{-1/4} \cdot \exp(w) =  |\Delta_F|^{-1/4} \cdot (e^{-w_1},e^{-w_2})$.
   The divisor $D$ translated from $W$ is $(O_F,u) \in \To$ with $u= e^{d/n} \cdot v = (e^{-w_1},e^{-w_2})$ and $ \|w\| = 10^{20}$ can be seen as the distance from $(O_F,1)$ to $D$.
 \end{example}

  \textbf{Input:}     
  \begin{itemize}
       \item  A basis for the lattice $O_F$: $\mathsf{a}_1 = (1,1)$ and $\mathsf{a}_2 = (-5 \time 10^{39}-0.5, 5 \time 10^{39}-0.5)$.
       \item  $v= |\Delta_F|^{-1/4} \cdot (e^{-w_1},e^{-w_2})$ and $(w_1, w_2)=(-\frac{10^{20}}{\sqrt{2}}, \frac{10^{20}}{\sqrt{2}})$.
       \item  $\delta \approx 10^{-5}$.
  \end{itemize}

We apply Algorithm \ref{alg} as follows.
\begin{arabiclist}
\item \textbf{Find a divisor $D'$ close to $D$ in $\Pic^0_F$.}\\                  
      The finite part of $W$ is $I=O_F$. As the notations in Section \ref{sec:jump}, $D_1$ is the zero divisor $(O_F,1)$, $D_2 = D$ and $u'=u$. We can skip part \ref{st1} and part \ref{st3} in Section \ref{sec:jump}. To find a divisor $D'$ obtained from some LLL reduction close to the $D$, it is sufficient to do part \ref{st2}, i.e., do Algorithm \ref{alg2}, as follows. 
       
     The smallest integer $t$ such that $n\cdot2^{-t}\cdot|w_i|<\log \partial_F $ for $i= 1, 2$ is $t = 61$. Let $(z_1, z_2) = 2^{-61} \cdot w \approx (-30.66587, 30.66587)$ and $\omega = (e^{-z_1},e^{-z_2}) $. 
        
   As described in Algorithm \ref{alg2}, denote by  $W_i=(O_F,\omega^{2^i})$ and $W_i = D'_i + (O_F, \omega_i)$ with $D'_i = d(J_i)$ a good divisor obtained from an LLL-reduction on $W_i$. By performing doubling and LLL-reduction 61 times, we can reach to $D$. The result can be seen in Table \ref{table1}.

In Table \ref{table1}, the second column contains the matrices $N_i$ for which $M_i = M_0\cdot N_i$ where $M_0$ and $M_i$ are the matrices of which columns form a basis of $O_F$ and $J_i$ respectively for all $ i= 1,2,...,61$. 

Let $J =J_{61} $ and $s=\omega_{61} \approx (e^{-0.80975}, e^{0.80975} )$. Then by choosing $D' = D'_{61}= d(J)$, we obtain that  $D = D' + (O_F, s) =  (J,s N(J)^{-1/n}) \in \Pic_F^0.$

\item \textbf{Apply Poisson summation formula.}\\
Since $D = (J,s N(J)^{-1/n}) \in \Pic_F^0$ and $e^{-d/n} \cdot N(J)^{-1/n} = (\co(J))^{-1/2}\approx 13.46966$, we obtain that $W= D+(O_F, e^{-d/n}) = \left(J, s (\co(J))^{-1/2} \right)$. The lattice associated to $W$ is $L= (\co(J))^{-1/n} s J$.

 \begin{alphlist}[(a)]
\item  $L$ has an LLL-reduced basis $\{b_{1}, b_2 \}$ with \\
 $b_{1}\approx(0.12124, 0.61234),  b_2 \approx (-1.57394, 0.29870)$. 
\item $b_{1}^* = b_1$, $b_{2}^* = (-1.57148, 0.31114)$,  $A_{1,1}=\|\mathsf{b}_1^*\|^2 \approx 0.38966 $ and $A_{2,2}=\|\mathsf{b}_2^*\|^2\approx 2.56635$, $A_{2,1}=\langle \mathsf{b}_2, b_1^* \rangle/\|b_1^*\|^2 \approx -0.02033$.

\item Since $\|\mathsf{b}_1\| < 1$ and $\|\mathsf{b}_2^*\| \ge 1$, we have $k=1$. Let $L_1 = \mathbb{Z} \cdot \mathsf{b}_1 \text{ and } L_2 = \mathbb{Z} \cdot \mathsf{b}_2$    as the notations in Section \ref{Poisson}. Then
$B=\left( \begin{array}{c}
        b_1
    \end{array} \right)
  =\left( \begin{array}{c}
          0.12124  \\
           0.61234 
              \end{array} \right)$.
\begin{romanlist}
   \item[(i)] $G=B^t \cdot B 
    =\left( \begin{array}{c}
            0.38966
                \end{array} \right)$ and
  $C=B \cdot G^{-1} =\left( \begin{array}{c}
            0.31115  \\
            1.57147
                \end{array} \right)$.
  
  \item[(ii)] $\mathsf{c}_1=(0.31115, 1.57147)$ is the column of $C$ and $\mathsf{c}_1^*=(0.31115, 1.57147)$, $C_{1,1}=\|\mathsf{c}_1^*\|^2=2.56635$.
  
  \item[(iii)] $\langle \mathsf{c}_1,\mathsf{b}_1^* \rangle =1 $  and $\langle \mathsf{c}_1,\mathsf{b}_2^* \rangle =0$.
  
  \item[(iv)] $Q_1(x_1, x_2) =    2.56635 x_1^2  + 2.56635 x_2^2$, 
  $Q_2(x_1, x_2) = - 0.02033 x_1 x_2 $ and \\
  $Q(x_1, x_2) =  Q_1(x_1, x_2)+ 2 Q_2(x_1, x_2) i$.\\
  Let $L'$ be the lattice associated to $Q_1(\mathsf{x})$.
  
\end{romanlist}

\end{alphlist}

\item \textbf{Find short vectors of the lattice $L'$.}
  \begin{alphlist}[(a)]
    \item 
    To approximate  $h^0(W)$ for quadratic fields with an error  $\delta \approx 10^{-5}$, it is sufficient to choose  $M =8$.
    \item The Fincke--Pohst method (Algorithm 2.12 in \cite{ref:40}) is used to find the list $\mathcal{L}_2$ of all columns vectors $\mathsf{x}=(x_1, x_2) \in \mathbb{Z}^2\backslash \{0,0\}$  such that $Q_1(\mathsf{x})= 2.56635 x_1^2  + 2.56635 x_2^2  \leq M $. Note that the function \texttt{qfminim} in  \texttt{pari-gp} can be used to find $\mathcal{L}_2$. Here $\mathcal{L}_2$ contains only 4 vectors (up to a sign).  
             $$ \mathcal{L}_2= \left( \begin{array}{cccc}
            0 & 1 & 1 & 1\\
            1 & 0 & 1 & -1
                \end{array} \right).$$

 By symmetry and since $Q(0,1) = Q(1,0)$, an approximate value of $h^0(W)$  is obtained as follows. 
 $$  h^0(W) \approx  \log{  \left(\frac{1}{\sqrt{0.38966}}   \left( 1+ 4 e^{-\pi Q(0,1)} + 2 e^{-\pi Q(1,1)} + 2 e^{-\pi Q(1,-1)} \right) \right)} $$
 $$  = 0.47250.$$         
    
   \end{alphlist}
   
\end{arabiclist}

    \textbf{Output:} $h^0(W) \approx 0.47250$ with an error $\delta = 10^{-5}$.

  Recall that the coset $\Pic_F^{(d)}$ is a circle containing the point $W$. Let $X=W$ and let $Y= W+ (O_F, v')$ be the points on $\Pic_F^{(d)}$ where $\log(v')=50 \cdot \mathsf{e}$. To see what $h^0$ looks like, we  compute $h^0$ at the points in the interval $[X,Y]$ on $\Pic_F^{d}$ and plot it. The points in this interval have corresponding translated divisors in the interval $[A, B] $ on $T^0$ where $A= \left(O_F,u \right)$ and $B= A+ \left(O_F,v'\right)$.
    
 First, the interval $[A, B]$ is divided into small intervals of length 1.  After that, we do LLL-reduction on the middle points of the small intervals to obtain good divisors. Let $S$ be the set of all good divisors obtained by this way. Then $S$ has 18 divisors in total (see Figure \ref{pic:Sn2}). 
  
 \begin{figure}[h]
    \centering
    \includegraphics[width=0.8\textwidth]{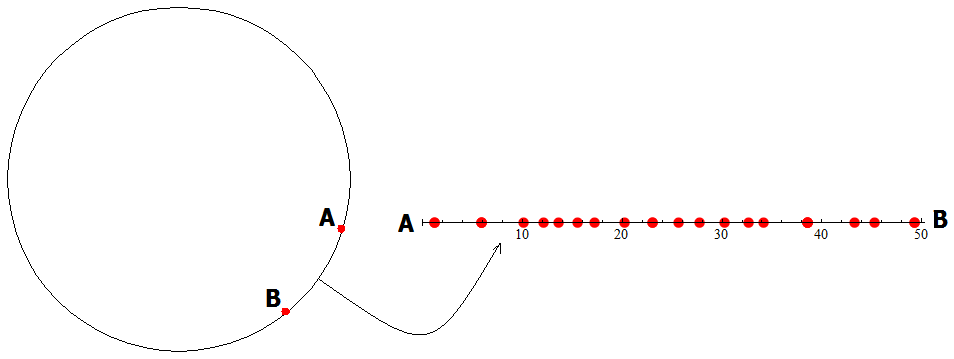}
    \caption{Good divisors in the interval $[A,B]$ on $\To$.}
    \label{pic:Sn2}
 \end{figure}

Now, let $W'$ be an arbitrary divisor in $[X, Y]$. Then its translated divisor $D$ is a divisor in $[A,B]$. We search for a good divisor $D'$ in $S$ which is the closest to $D$ and use $D'$ to compute $h^0(W')$.  Then $h^0$ is plotted in the interval $[X, Y]$ as in Figures \ref{pic:n2h2} below. In which the red points are divisors on $\Pic_F^{(d)}$ whose translated divisors are divisors in $S$. 
       
\begin{figure}[h]
   \centering
   \includegraphics[width=0.8\textwidth]{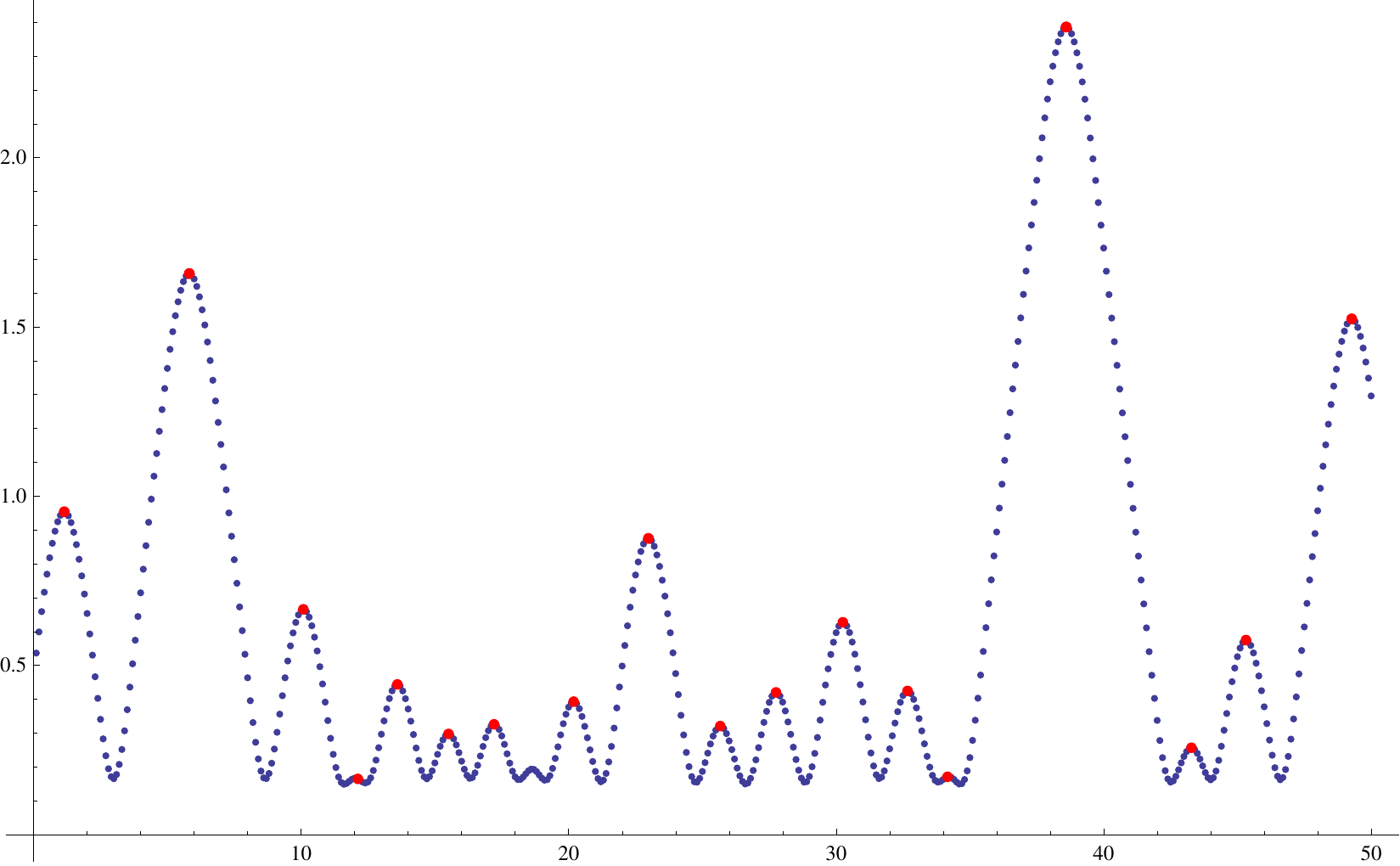}
   \caption{$h^0$ for the quadratic field $Q(\sqrt{10^{80}+129})$.}
   \label{pic:n2h2}
\end{figure}

\begin{example}\label{ex:2} 
  Let $P = X^3 - 88998X^2 - 1090173446X - 1000470997815$. Then $P$ is an irreducible polynomial with 3 real roots denoted by $\alpha_1 = \alpha, \alpha_2, \alpha_3$. Let $F = \mathbb{Q}(\alpha) $. Thus, $F$ is a real cubic field with 3 real infinite  primes
    $\sigma_i:  F \longrightarrow \mathbb{R}^3$ sending $\alpha$ to  $\alpha_i$ for $ i = 1, 2, 3$. The discriminant of $F$ is $\Delta_F = 10000820940380105429207549453 > 10^{28}$.
    
  We have $\left(\mathbb{R}^3\right)^0 = \{  (x_1,x_2,x_3) \in \mathbb{R}^3: x_1 + x_2 +  x_3 =0 \}$ a 2-dimensional subspace of $\mathbb{R}^3$ with an orthonormal basis $\{\mathsf{e}_1 = (\frac{1}{\sqrt{2}}, 0, -\frac{1}{\sqrt{2}}), \mathsf{e}_2 = (\frac{1}{\sqrt{6}}, -\frac{2}{\sqrt{6}}, \frac{1}{\sqrt{6}})\}$.
  
  The connected component of identity of $\Pic_F^0$ is the Dirichlet torus $\To = \left(\mathbb{R}^3\right)^0\slash \Lambda$ where $\Lambda$ is the lattice $ \{(\log{|\sigma_1(f)|},\log{|\sigma_2(f)|}, \log{|\sigma_3(f)|)}: f \in O_F^* \}$.

 Denote by 
 $w = 10^{10} \cdot \mathsf{e}_1 + 10^{10} \cdot \mathsf{e}_2$\\
 $ \approx (11153550716.50411, -8164965809.27726,$ $ -2988584907.22685) =(w_1, w_2, w_3) $. \\
 Let $W \in \Pic_F^{(d)}$ where $v= |\Delta_F|^{-1/6} \cdot \exp(w) =  |\Delta_F|^{-1/6} \cdot (e^{-w_1},e^{-w_2}, e^{-w_3})$. The divisor $D$ translated from $W$ is $(O_F,u) \in \To$ with $u= e^{d/n} \cdot v = (e^{-w_1},e^{-w_2}, e^{-w_3})$ and  $  \|w \| \approx 14142135623.73095$ can be considered as the distance from $(O_F,1)$ to $D$.
  \end{example}

    \textbf{Input:}     
    \begin{itemize}
         \item  A basis for the lattice $O_F$: $\mathsf{a}_1 = (1,1,1)$, \\
          $\mathsf{a}_2 \approx (-39667.47830, -30666.32509,70333.80339),\\
           \mathsf{a}_3 \approx (451263681.71217, -491488332.71468, 40224652.00251) $.
         \item  $v = |\Delta_F|^{-1/6}\cdot(e^{-w_1},e^{-w_2}, e^{-w_3})$ where\\
          $(w_1, w_2, w_3) \approx (11153550716.50411, -8164965809.27726, -2988584907.22685)$.
         \item  $\delta \approx 10^{-5}$.
    \end{itemize}

   We  compute $h^0(W)$ by using Algorithm \ref{alg}.
   
\begin{arabiclist}
\item \textbf{Find a divisor $D'$ close to $D$ in $\Pic^0_F$.}\\
As the notations in Section \ref{sec:jump}, $D_1 =(O_F,1)$ the zero divisor, $D_2 = D$ and $u'=u$. Similar to Example \ref{ex:1}, we can skip part \ref{st1} and part \ref{st3}. To find a divisor $D'$ obtained from some LLL reduction close to the $D$, it is sufficient to do part \ref{st2}, i.e., do Algorithm \ref{alg2}, as follows.

 We have that $t =30$  is the smallest integer for which $n\cdot2^{-t}\cdot|w_i|<\log \partial_F$ for $i= 1, 2, 3$. Let $(z_1, z_2, z_3) = 2^{-30}\cdot w \approx (10.38755, -7.60422, -2.78333) $ and $\omega = (e^{-z_1},e^{-z_2},e^{-z_3}) $.              
  
 As described in Algorithm \ref{alg2}, denote by  $W_i=(O_F,\omega^{2^i})$ and $W_i = D'_i + (O_F, \omega_i)$ with $D_i' = d(J_i)$, a good divisor obtained from an LLL-reduction on $D_i$.
      
      In Table \ref{table2}, the second column contains the matrices of which columns form an LLL-reduced basis for the lattices $J_i$ for all $ i= 1,2,...,30$. The coordinates of these vectors are computed with respect to the basis of $O_F$.

  Let $J =J_{30} $ and $s=\omega_{30} \approx (e^{1.09384}, e^{-0.61605}, e^{-0.47780} )$. Then by choosing $D' = D'_{30}= d(J)$, we obtain that $D= D' +(O_F,s)= (J,s N(J)^{-1/n}) \in \Pic_F^0.$

  \item \textbf{Apply Poisson summation formula.}\\
  Since $D = (J,s N(J)^{-1/n}) \in \Pic_F^0$ and $e^{-d/n} \cdot N(J)^{-1/n} = (\co(J))^{-1/2}$, it follows that $W= D+(O_F, e^{-d/n}) = \left(J, s (\co(J))^{-1/2} \right)$. 
   Here $N(J)^{-1} = 938139713086$ and  $\co(J) \approx 106.59831$.
   The lattice associated to $W$ is $L= (\co(J))^{-1/n} sJ$.

  \begin{alphlist}[(a)]
  \item  $L$ has an LLL-reduced basis $\{\mathsf{b}_{1}, \mathsf{b}_2, \mathsf{b}_3 \}$ with\\
  $\mathsf{b}_{1} \approx (0.07064, 0.39051, 0.34009), \mathsf{b}_2 \approx (0.83795, 0.53263,-0.87506),$\\
  $\mathsf{b}_3 \approx (1.35900, -0.65069, 0.22025)$. 
  \item $b_{1}^* = b_1$, 
  $b_{2}^* = (0.84581, 0.57611, -0.83720)$ and\\
   $b_{3}^* = (1.09497, -0.72624, 0.60648 )$.\\
    $A_{1,1}=\|\mathsf{b}_1^*\|^2 \approx 0.27314 $, $A_{2,2}=\|\mathsf{b}_2^*\|^2\approx 1.74820$, $A_{3,3}=\|\mathsf{b}_3^*\|^2\approx 2.09420$ and 
    $A_{2,1}=\langle \mathsf{b}_2, b_1^* \rangle/\|b_1^*\|^2 \approx -0.11133$,
    $A_{3,1}=\langle \mathsf{b}_3, b_1^* \rangle/\|b_1^*\|^2 \approx -0.30460$ and    
    $A_{3,2}=\langle \mathsf{b}_3, b_2^* \rangle/\|b_2^*\|^2 \approx 0.33761$.  
  
 Let $L_1 = \mathbb{Z} \cdot \mathsf{b}_1  \text{ and } L_2 = \mathbb{Z} \cdot \mathsf{b}_2 \oplus \mathbb{Z} \cdot \mathsf{b}_3$. 
    
  \item Since $\|\mathsf{b}_1\| < 1$ and $\|\mathsf{b}_2^*\| \ge 1$, we have $k=1$ and 
  $B=\left( \begin{array}{c}
          b_1
      \end{array} \right)
    =\left( \begin{array}{c}
            0.07064\\ 0.39051\\ 0.34009
      \end{array} \right)$.
   \begin{romanlist}
      \item[(i)] $G=B^t \cdot B 
          =\left( \begin{array}{c}
                  0.27314
                      \end{array} \right)$ and
        $C=B \cdot G^{-1} =\left( \begin{array}{c}
                  0.258609 \\
                  1.42968\\
                  1.24508
                      \end{array} \right)$.
       \item[(ii)] $\mathsf{c}_1=(0.258609, 1.42968, 1.24508)$ is the column of $C$ and\\ $\mathsf{c}_1^*=(0.258609, 1.42968, 1.24508)$, $C_{1,1}=\|\mathsf{c}_1^*\|^2=3.66108$.
       
       \item[(iii)] $\langle \mathsf{c}_1,\mathsf{b}_1^* \rangle =1 $  and $\langle \mathsf{c}_1,\mathsf{b}_2^* \rangle=\langle \mathsf{c}_1,\mathsf{b}_3^* \rangle =0$.
         
       \item[(iv)] Denote by $Q_1(x_1, x_2, x_3) = 3.66108 x_1^2 + 1.74820(x_2 + 0.33761 x_3)^2 + 2.09420 x_3^2$, 
           $Q_2(x_1, x_2, x_3) = 0.11133 x_1 x_2 + 0.30460 x_1 x_3$
            and  $ Q(x_1, x_2, x_3) = Q_1(x_1, x_2, x_3) + 2 Q_2(x_1, x_2, x_3) i$.\\
            Let $L'$ be the lattice associated to $Q_1(\mathsf{x})$.

   \end{romanlist} 
 
  \end{alphlist}

 \item \textbf{Find short vectors of the lattice $L'$.}
  \begin{alphlist}[(a)]
      \item By Proposition \ref{esterr}, to approximate  $h^0(W)$ for cubic fields with an error  $\delta \approx 10^{-5}$, it is sufficient to choose  $M =9$.
      \item The Fincke--Pohst method (Algorithm 2.12 in \cite{ref:40}) is used to find the list $\mathcal{L}_2$ of all column vectors $\mathsf{x}=(x_1, x_2, x_3) \in \mathbb{Z}^3\backslash \{0,0,0\}$  such that $ Q_1(\mathsf{x}) \leq M $. Here $\mathcal{L}_2$ has 17 vectors (up to a sign) obtained by using the function \texttt{qfminim} in  \texttt{pari-gp}.
           $$ \mathcal{L}_2= \left( \begin{array}{cccccccccccccccccc}
           0 & 0 & 0 & 0 &  0 &  0 &  0 & 0 & 1 & 1 &  1 & 1 & 1 & 1 &  1 &  1 &  1\\
           1 & 2 & 0 & 1 & -1 & -2 & -1 & 0 & 0 & 1 & -1 & 0 & 1 &-1 &  0 &  1 & -1\\
           0 & 0 & 1 & 1 &  1 &  1 &  2 & 2 & 0 & 0 &  0 & 1 & 1 & 1 & -1 & -1 & -1
              \end{array} \right).$$

 By symmetry, an approximate value of $h^0(W)$  is obtained as follows. 
 $$  h^0(W) \approx  \log{  \left( \frac{1}{\sqrt{ 0.27314}} \left(1+ 2 \sum_{\mathsf{x} \in \mathcal{L}_2}e^{-\pi Q(\mathsf{x})}\right) \right)} = 0.65882.$$       
       
  \end{alphlist}     
   
\end{arabiclist}

\textbf{Output:} $h^0(W) \approx 0.65882$ with error $\delta = 10^{-5}$.

 Putting $v_1= 1$, $v_2= \exp{(6 \cdot\mathsf{e}_1)}$, $v_3= \exp{(6 \cdot\mathsf{e}_1+ 8\cdot \mathsf{e}_2)}$ and $v_4= \exp{(8\cdot \mathsf{e}_2)}$. 
 Let $X_i= W+ (O_F, v_i)$ be the points on $\Pic_F^{(d)}$ for $i=1,2,3, 4$. 
 Similar to Example \ref{ex:1}, to see what $h^0$ looks like, we compute $h^0$ at the points in the box $X_1 X_2 X_3 X_4$ on $\Pic_F^{d}$ and plot it. 
 The points in this box have corresponding translated divisors in the rectangle $ABCE$ on $T^0$. Here $A=(O_F,u), B= A+ (O_F, v_2), C= A+ (O_F, v_3)$ and  $E= A+ (O_F, v_4)$ .  
    
 \begin{figure}[h] 
     \centering
     \includegraphics[width=0.3\textwidth]{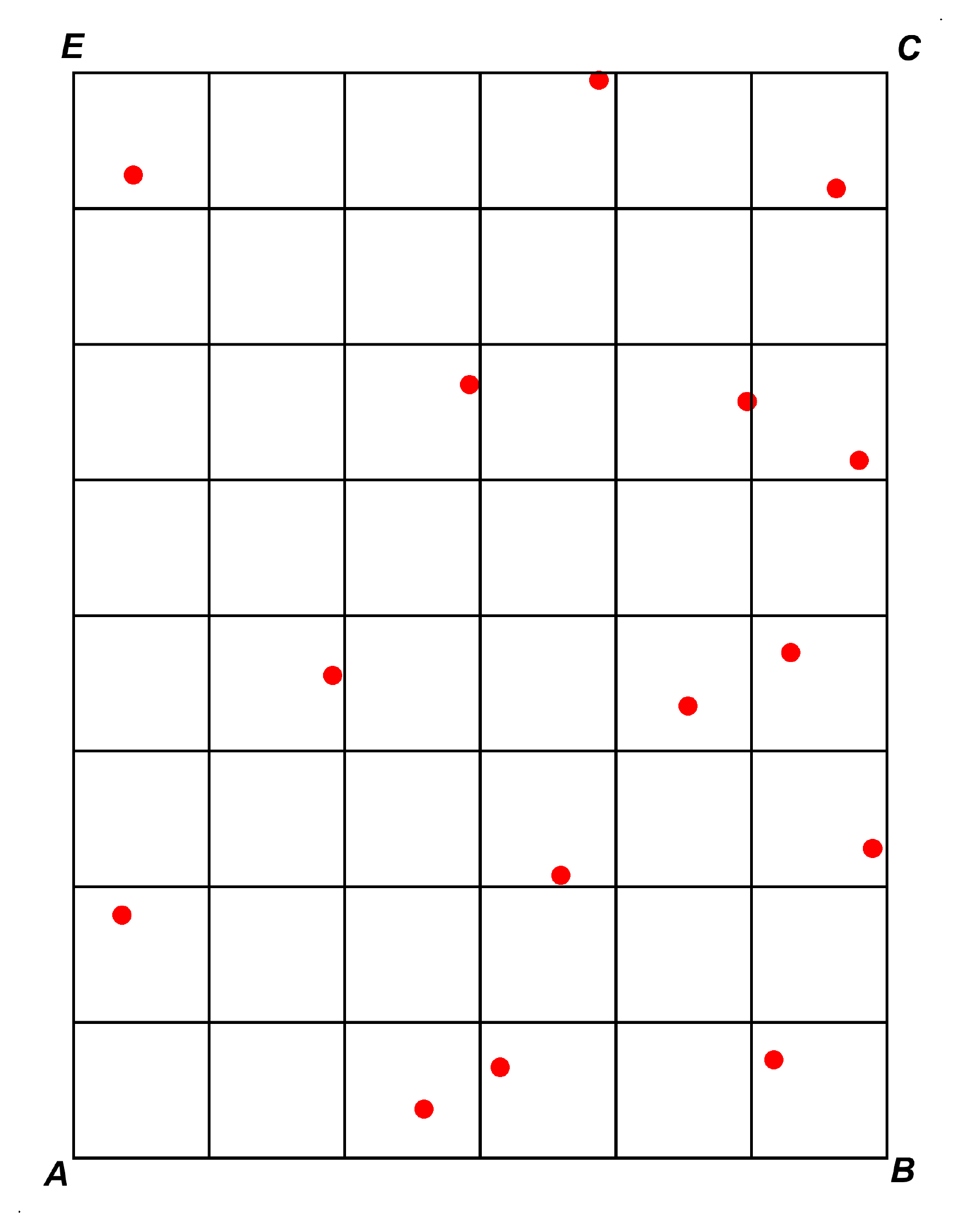}
     \caption{The set $S$ of good divisors in the rectangle $ABCE$.}
      \label{pic:h0}
\end{figure}

  The rectangle $ABCE$ is divided into small squares, each one has sides of length 1. After that, we perform LLL-reduction at the center of such squares to obtain good divisors. Let $S$ be the set of all good divisors obtained by this way. Then $S$ has 15 points in total (see Figure \ref{pic:h0}). 
      
  Now, let $W'$ be an arbitrary divisor in $X_1 X_2 X_3 X_4$. Then its translated divisor $D$ is a divisor in $ABCE$. We search for a good divisor $D'$ in $S$ which is the closest to $D$ and use $D'$ to compute $h^0(W')$.  Then $h^0$ is plotted as in Figure \ref{pic:h1} and \ref{pic:h2} below 
 in which the red points are divisors on $\Pic_F^{(d)}$ whose translated divisors are good divisors in $S$. In Figure \ref{pic:h1}, the dark color area is corespondent to the large values of $h^0$.

 \begin{figure}[h]
       \centering
       \includegraphics[width=0.8\textwidth]{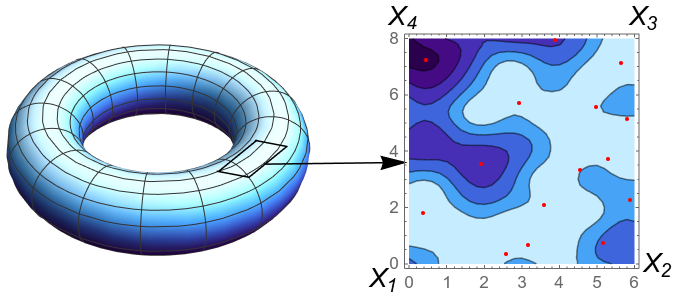}
        \caption{The contour plot of $h^0$ in the box $X_1 X_2 X_3 X_4$.}
         \label{pic:h1}
 \end{figure}

\begin{figure}[h]
   \centering
   \includegraphics[width=0.6\textwidth]{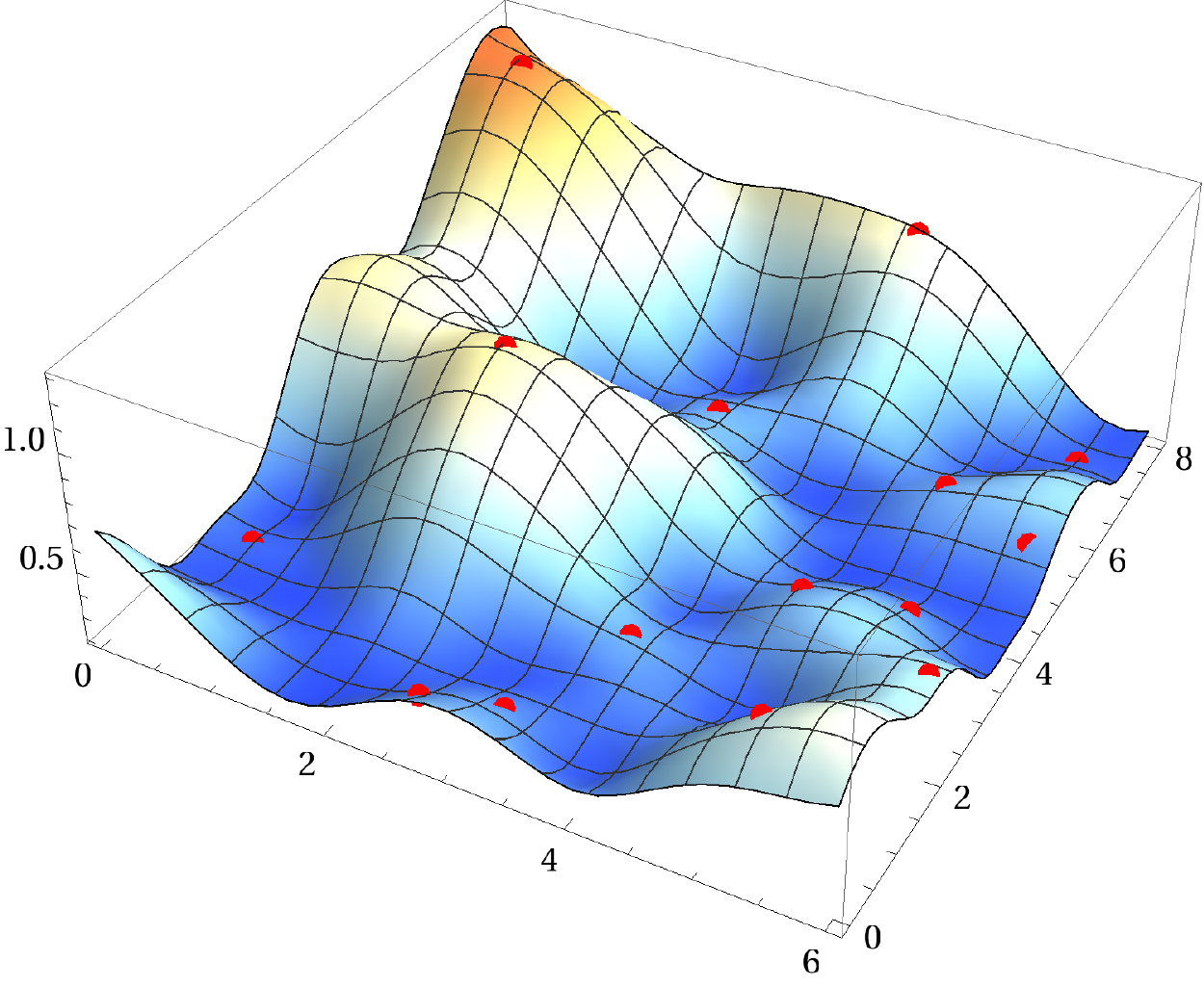}
   \caption{$ h^0$ for the real cubic field $F$.}
    \label{pic:h2}
\end{figure}

\begin{landscape}
        \captionof{table}{}
        \label{table1}
         \resizebox{1.5\textwidth}{!}{
               \begin{tabular}{|c c c c|} 
               	\hline 
               	$i$ & $J_i$ & $N(J_i^{-1})$ & $\log{\omega_i}$ \\ [0.5ex] 
               
               	\hline
               	\hspace*{0.5cm} & \hspace*{0.5cm} & \hspace*{0.5cm}& \hspace*{0.5cm} \\
               	0 & $ O_F$  & 1 & $(-30.66587, 30.66587)$ \\ 
               	
               	\hline
               	\hspace*{0.5cm} & \hspace*{0.5cm} & \hspace*{0.5cm}& \hspace*{0.5cm} \\
               	1 & $\frac{1}{N(J_1^{-1})} \left( \begin{array}{cc}    	
               	129 & 5\cdot 10^{39}+64 \\
               	0 & -1  \\   	
               	\end{array} \right)$ & 129 & $(29.03491,-29.03491)$ \\
               	
               	\hline
               	\hspace*{0.5cm} & \hspace*{0.5cm} & \hspace*{0.5cm}& \hspace*{0.5cm} \\
               	2 & $\frac{1}{N(J_2^{-1})} \left( \begin{array}{cc}
               	129 & 5\cdot 10^{39}+25 \\
               	0 & 1  \\
               	\end{array} \right)$ & 129  & $(-32.29683,32.29683)$ \\
               	
               	\hline
               	\hspace*{0.5cm} & \hspace*{0.5cm} & \hspace*{0.5cm}& \hspace*{0.5cm} \\
               	3 & $\frac{1}{N(J_3^{-1})} \left( \begin{array}{cc}
               	2146689 & 5\cdot 10^{39}+261160 \\
               	0 & -1  \\
               	\end{array} \right)$ & 2146689 & $(25.77299,-25.77299)$ \\
               	\hline
               	
               	\hspace*{0.5cm} & \hspace*{0.5cm} & \hspace*{0.5cm}& \hspace*{0.5cm} \\
               	\vdots &  \vdots &  \vdots&  \vdots \\ [1ex] 
               	\hspace*{0.5cm} & \hspace*{0.5cm} & \hspace*{0.5cm}& \hspace*{0.5cm} \\
               	\hline

               	\hspace*{0.5cm} & \hspace*{0.5cm} & \hspace*{0.5cm}& \hspace*{0.5cm} \\
               	60 & $ \left( \begin{array}{cc}
               	1 & -\frac{3535142278155418356810286962804039637657}{N(J_{60}^{-1})}  \\ 
               	\hspace*{0.5cm} &  \hspace*{0.5cm} \\
               	    	0 & -\frac{1}{N(J_{60}^{-1})}  \\
               	\end{array} \right)$ & 19902657321594605283368410 59638321226594  & $(0.64301,-0.64301)$ \\
               	\hline

               	\hspace*{0.5cm} & \hspace*{0.5cm} & \hspace*{0.5cm}& \hspace*{0.5cm} \\
               	61 & $ \left( \begin{array}{ccc}
               1 & -\frac{4637851969685051510247636695057652362226}{N(J_{61}^{-1})}  \\ 
               		\hspace*{0.5cm} &  \hspace*{0.5cm} \\
               	0 & \frac{1}{N(J_{61}^{-1})}  \\
               	\end{array} \right)$ & 742409068975056334669660076059992618001    & $(0.80975,-0.80975)$ \\
               	\hline
               \end{tabular}
               }
 \end{landscape}

\begin{landscape}
      \captionof{table}{}
      \label{table2}
       \resizebox{1.38\textwidth}{!}{
        	\begin{tabular}{|c c c c|} 
        		\hline
        		$i$ & $J_i$ & $N(J_i^{-1})$ & $\log{\omega_i}$ \\ [0.5ex] 
        		\hline 
        		
        		0 & $ O_F$  & 1 & $(10.38755, -7.60422, -2.78334)$ \\ 
        		\hline
        		\hspace*{0.5cm} & \hspace*{0.5cm} & \hspace*{0.5cm}& \hspace*{0.5cm} \\
        		1 & $\frac{1}{N(J_1^{-1})} \left( \begin{array}{ccccc}
        		691582920399 & 118502077239 & -227460374946\\  		
        		0 & 40389624 & 42667155\\  		
        		0 & 3277 & -13661\end{array} \right)$ & 691582920399 & $(1.05601, -1.5469, 0.49083)$ \\
        		\hline

        		\hspace*{0.5cm} & \hspace*{0.5cm} & \hspace*{0.5cm}& \hspace*{0.5cm} \\
        		2 & $\frac{1}{N(J_2^{-1})} 
        		\left( \begin{array}{ccccc}
        		222208932162 & 82987136358 & 68645452464 \\
        	  	0 & 19963734 & 42669156 \\
        	  	0 & 4774 &  -927\end{array} \right)$ & 222208932162 & $(-0.42917,- 0.37727, 0.80643)$ \\
        		\hline

        		\hspace*{0.5cm} & \hspace*{0.5cm} & \hspace*{0.5cm}& \hspace*{0.5cm} \\
       			3 & $\frac{1}{N(J_3^{-1})} 
       		  		\left( \begin{array}{ccccc}
       		  		4204248079595 & 1034415034603 & 307388467818 \\
       		  	  	0 & 60087131 & 218999921 \\
       		  	  	0 & 14593 &  -16782\end{array} \right)$ & 4204248079595 & $(0.03212, 0.32161,- 0.35373)$ \\
        		\hline  		
        	    		
       			\hspace*{0.5cm} & \hspace*{0.5cm} & \hspace*{0.5cm}& \hspace*{0.5cm} \\
       			\vdots &  \vdots &  \vdots&  \vdots \\ [1ex] 
       			\hspace*{0.5cm} & \hspace*{0.5cm} & \hspace*{0.5cm}& \hspace*{0.5cm} \\
       			 \hline

        		\hspace*{0.5cm} & \hspace*{0.5cm} & \hspace*{0.5cm}& \hspace*{0.5cm} \\
        		29 & $\frac{1}{N(J_{29}^{-1})} \left( \begin{array}{ccccc}
        		3063324517380 & 662309451492 & -144083620104 \\
        		0 & 132428244 & 106307652\\
        		0 & 10287 & -14874 \end{array} \right)$ & 3063324517380 & $(-0.99628, -0.39485, 1.3911)$ \\ [1ex] 
        		\hline

        		\hspace*{0.5cm} & \hspace*{0.5cm} & \hspace*{0.5cm}& \hspace*{0.5cm} \\
        		30 & $\frac{1}{N(J_{30}^{-1})} \left( \begin{array}{ccccc}
        		  		938139713086 & -448140560830 & 27116000512 \\
        		  		0 & 3374186 & -90561028\\
        		  		0 & 10388 & -773 \end{array} \right)$ & 938139713086 & $(1.09384, -0.61605, -0.47780)$ \\ [1ex] 
        		\hline

      \end{tabular}
        	
        	}
\end{landscape}

\section*{Acknowledgement}
I would like to thank  Ren\'{e} Schoof for proposing the Poisson summation formula for computing $h^0$ and for very valuable comments. The author also would like to thank the reviewers for their insightful comments that helped improve the manuscript. I am also immensely grateful to Dave Karpuk and Jaana Suviniitty for their useful comments. 

This research was supported by the Universit\`{a} di Roma ``Tor Vergata" and partially supported by the Academy of Finland (grants  $\#$276031, $\#$282938, and $\#$283262). The support from the European Science Foundation under the COST Action IC1104 is also gratefully acknowledged.


\end{document}